\documentclass[10pt]{article}
\usepackage[T1]{fontenc}
\usepackage[latin1]{inputenc}
\usepackage{verbatim}
\usepackage{amsmath}
\usepackage{amscd}
\usepackage{graphicx}
\usepackage{amsfonts}
\usepackage{amssymb}
\makeatletter
\makeatletter
\providecommand{\U}[1]{\protect\rule{.1in}{.1in}}
\providecommand{\U}[1]{\protect\rule{.1in}{.1in}}
\newtheorem{theorem}{Theorem}

\newtheorem{corollary}[theorem]{Corollary}

\newtheorem{definition}[theorem]{Definition}

\newtheorem{lemma}[theorem]{Lemma}
\newtheorem{notation}[theorem]{Notation}

\newtheorem{proposition}[theorem]{Proposition}
\newtheorem{remark}[theorem]{Remark}

\newtheorem{blanket}[theorem]{Blanket Assumption}
\newenvironment{proof}[1][Proof]{\noindent\textbf{#1.} }{\ \rule{0.5em}{0.5em}}
\makeatother

\makeatother

\begin{document}

\title{Groupoid and Inverse Semigroup Presentations\\of\\Ultragraph $C^{*}$-Algebras}
\author{Alberto E. Marrero\thanks{Supported by grants from the National Science
Foundation and the Sloan Foundation and by a GAANN Fellowship}\ and Paul S.
Muhly\thanks{Supported by a grant from the National Science Foundation
(DMS-0355443).}\medskip{}\\Department of Mathematics-Physics\\University of Puerto Rico\\Cayey, PR 00736 \medskip{} \\Department of Mathematics\\The University of Iowa\\Iowa City, IA 52242\medskip{} \\alberto\_marrero@yahoo.com\\pmuhly@math.uiowa.edu\\}
\maketitle
\begin{abstract}
Inspired by the work of Paterson on $C^{\ast}$-algebras of directed graphs, we
show how to associate a groupoid $\mathfrak{G}_{\mathcal{G}}$ to an ultragraph
$\mathcal{G}$ in such a way that the $C^*$-algebra of $\mathfrak{G}_{\mathcal{G}}$ is canonically isomorphic to Tomforde's $C^*$-algebra $C^*(\mathcal{G})$.  The groupoid $\mathfrak{G}_{\mathcal{G}}$ is built from an inverse semigroup $S_{\mathcal{G}}$ naturally associated to $\mathcal{G}$. 
\end{abstract}

\section{{\protect\small INTRODUCTION}}

Cuntz and Krieger described a way to associate a $C^{\ast}$-algebra
$\mathcal{O}_{A}$ to a finite square matrix $A$ with entries in $\left\{
0,1\right\}  $ in \cite{CK80}. Subsequently, their work was set in the context
of graphs by a number of authors (see, e.g. \cite{yW80}). It was soon
recognized that infinite graphs led to problems that were not covered in
\cite{CK80} and numerous attempts over the years have been advanced for
dealing with them. One very interesting attempt was introduced by Mark
Tomforde in \cite{MTT}, where he defined the notion of an \emph{ultragraph}.
Roughly speaking, an ultragraph is a generalization of directed graph in which
the range of an edge is allowed to be a \emph{set} of vertices rather than
just a single vertex. In \cite{MTT} and \cite{MTPS}, Tomforde showed that the
class of ultragraph $C^{\ast}$-algebras includes all graph $C^{\ast}$-algebras
and all so-called Exel-Laca algebras, as well as $C^{\ast}$-algebras that are
in neither of these classes. Our goal in this note is to determine a groupoid
model $\mathfrak{G}_{\mathcal{G}}$ for an ultragraph $\mathcal{G}$ in such a
way that $C^{\ast}\left(  \mathcal{G}\right)  \backsimeq C^{\ast}\left(
\mathfrak{G}_{\mathcal{G}}\right)  $, revealing salient features of $C^{\ast
}(\mathcal{G})$. The groupoid connection enables one to interpret properties
of $C^{\ast}\left(  \mathcal{G}\right)  $ in dynamical terms. In the directed
graph setting, this has been done with considerable consequence in~\cite{kprr}
and \cite{APP}.

Our approach is inspired by Paterson's paper \cite{APP}. We first build an
inverse semigroup $S_{\mathcal{G}}$ that is designed to reflect the
representation theory of $\mathcal{G}$. A representation of $\mathcal{G}$ is
determined by certain partial isometries on a Hilbert space indexed by
vertices and edges from $\mathcal{G}$. The $C^{\ast}$-algebra $C^{\ast}\left(
\mathcal{G}\right)  $ of $\mathcal{G}$ is the universal $C^{\ast}$-algebra for
such representations \cite[Theorem 2.11]{MTT}. A representation of
$\mathcal{G}$ may be viewed directly as a representation of $S_{\mathcal{G}}$
by partial isometries. The groupoid model $\mathfrak{G}_{\mathcal{G}}$ we
construct is built from $S_{\mathcal{G}}$ based on the approach developed by
Paterson and exposed in his book \cite{APB}. We first build the universal
groupoid canonically associated to $S_{\mathcal{G}}$ and then take a certain
reduction for $\mathfrak{G}_{\mathcal{G}}$. We will call $\mathfrak
{G}_{\mathcal{G}}$ the \emph{ultrapath groupoid} $\emph{of}$ $\mathcal{G}$.

For a bit more detail to help with the motivation, recall that if $E$ is an
ordinary directed graph (but not necessarily row finite or without sinks),
Paterson's inverse semigroup $S_{E}$, is the set of all pairs $\left(
\alpha,\beta\right)  $, where $\alpha,\beta$ are finite paths in the graph $E$
and $r\left(  \alpha\right)  =r\left(  \beta\right)  $, together with a zero
element $z$ \cite{APP}. The multiplication in $S_{E}$ is defined as follows:
$(\alpha,\alpha^{\prime}\mu)(\alpha^{\prime},\beta):=(\alpha,\beta\mu)$,
$(\alpha,\alpha^{\prime})(\alpha^{\prime}\mu,\beta^{\prime}):=(\alpha\mu
,\beta^{\prime})$ and all other products are the zero $z$. The involution on
$S_{E}$ is transposition: $(\alpha,\beta)^{\ast}:=(\beta,\alpha)$. In an
ordinary graph the paths of length zero are just the vertices. However in the
ultragraph case, the paths of length zero are the \emph{sets} in a space that we denote by
$\mathcal{G}^{0}$, which is defined to be the smallest subcollection of
subsets of $G^{0}$, that contains $\left\{  v\right\}  $, for all $v\in G^{0}%
$, contains $r\left(  e\right)  $ for all $e\in\mathcal{G}^{1}$, and is closed
under finite union and intersections. That is, $\mathcal{G}^{0}$ is the
{}``lattice\textquotedblright\ generated by $\{\{ v\}\mid v\in G^{0}\}$ and
the sets $r(e)$, $e\in\mathcal{G}^{1}$. We place {}``lattice\textquotedblright
\ in double quotes because $\mathcal{G}^{0}$ may fail to be a lattice in the
usual sense in that it may fail to contain all of $G^{0}$. Roughly speaking,
we think of enlarging $G^{0}$ by adding in additional vertices, one for each
\emph{set} $r(e)$, $e\in\mathcal{G}^{1}$, so that the elements of $\{ v:v\in
G^{0}\}\cup\{ r\left(  e\right)  :e\in\mathcal{G}^{1}\}$ play the role of
\ {}``generalized vertices\textquotedblright.\ Then the {}%
``lattice\textquotedblright\ $\mathcal{G}^{0}$ plays the role of \ {}``subsets
of generalized vertices\textquotedblright. Thus, since $C^{\ast}\left(
\mathcal{G}\right)  $ involves partial isometries indexed by those special
sets, we introduce the set $\mathfrak{p}$ consisting of $\mathcal{G}^{0}$
together with the set of all pairs $\left(  \alpha,A\right)  $, where $\alpha$
is a finite path in $\mathcal{G}$ with positive length, and $A\in
\mathcal{G}^{0}$, with $A\subseteq r\left(  \alpha\right)  $. That is,
$\alpha=\alpha_{1}\alpha_{2}\cdots\alpha_{l}$ with $s(\alpha_{i+1})\in
r(\alpha_{i})$. We will call $\mathfrak{p}$ the \textit{ultrapath space} of the
ultragraph\textit{\ }$\mathcal{G}$ . The range map $r$ and the source map $s$
extend to $\mathfrak{p}$ in a natural way. Our inverse semigroup
$S_{\mathcal{G}}$ is the set of pairs $(x,y)\in\mathfrak{p}\times\mathfrak{p}$
such that $r(x)=r(y)$. The operations on $S_{\mathcal{G}}$ are defined
similarly to those on $S_{E}$. However, the structure of $S_{\mathcal{G}}$ is
rather more complicated and much of our analysis is devoted to keeping track
of the complications.

The next section is devoted to the constructions of $S_{\mathcal{G}}$ and
$\mathfrak{G}_{\mathcal{G}}$ and to showing that $C^{\ast}(\mathcal{G})\simeq
C^{\ast}(\mathfrak{G}_{\mathcal{G}})$. In the subsequent section we address
the amenability of $C^{\ast}\left(  \mathcal{G}\right)  $. We use a groupoid
crossed product argument to show that $C^{\ast}(\mathfrak{G}_{\mathcal{G}})$
is nuclear and hence that $C^{\ast}(\mathfrak{G}_{\mathcal{G}})$ and $C^{\ast
}\left(  \mathcal{G}\right)  $ are amenable. The last section deals with the
\emph{simplicity} of $C^{\ast}\left(  \mathcal{G}\right)  $. We define an
analogue of the so-called {}``condition (K)\textquotedblright\ that appears in
the analysis of ordinary graph $C^{\ast}$-algebras. We show that $\mathfrak
{G}_{\mathcal{G}}$ is essentially principal if and only if $\mathcal{G}$
satisfies condition (K). When $\mathcal{G}$ satisfies condition (K), then thanks to the amenability of $\mathfrak{G}_{\mathcal{G}}$, the (norm closed, two sided) ideals in $C^{\ast}(\mathcal{G})$ and
$C^{\ast}(\mathfrak{G}_{\mathcal{G}})$ are parameterized by the open invariant subsets of the unit space of $\mathfrak{G}_{\mathcal{G}}$.  In particular, $C^{\ast}(\mathcal{G})$ and
$C^{\ast}(\mathfrak{G}_{\mathcal{G}})$ are simple if and only if $\mathfrak
{G}_{\mathcal{G}}$ is minimal.

\subsection{{\protect\small Notation and Conventions}%
{\protect\normalsize \label{Notation}}}

We set up here the basic notation we shall use for graphs, ultragraphs and
inverse semigroups. Additional notation will be developed as needed, below.

A directed graph\textbf{\ }$E=\left(E^{0},E^{1},r,s\right)$ consists of a
countable set of vertices $E^{0}$, a countable set of edges $E^{1}$, and maps
$r$, $s:E^{1}\rightarrow E^{0}$ identifying the range and source of each edge.
The graph $E$ is called \emph{row finite}\textit{\ }if for each $v\in E^{0}$,
the set of edges starting at $v$ is finite. The graph is called \emph{locally
finite} if for each vertex $v\in E^{0}$, the set of edges starting at $v$ is
finite and the set of edges terminating at $v$ is also finite. A vertex $v$ is
called a \emph{sink}, if there no edges starting at $v$. A\emph{\ finite
path}\textit{\ }is a sequence $\alpha$ of edges $e_{1}\ldots e_{k}$ where
$s\left(  e_{i+1}\right)  =r\left(  e_{i}\right)  $ for $1\leq i\leq k-1$. We
write $\alpha=e_{1}\ldots e_{k}$. The length $l\left(  \alpha\right)  $ of
$\alpha$ is just $k$. Each vertex $v$ is regarded as a finite path of length
zero. The source and range maps extend to the set of finite paths in the
natural way. The set of infinite paths, is the set of infinite sequences of
edges $\gamma=e_{1}e_{2}\ldots$, such that $s\left(  e_{i+1}\right)  =r\left(
e_{i}\right)  $ for all $i$. The source map extends to that set in the natural
way as well.

Following \cite{MTT}, an \emph{ultragraph} is a system $\mathcal{G}\mathbb{=}\left(  G^{0}%
,\mathcal{G}^{1},r,s\right)  $, where $G^{0}$ and $\mathcal{G}^{1}$ are
countable sets, called, respectively, the \emph{vertices} and \emph{edges} of
$\mathcal{G}$; where $s$ is a function from $\mathcal{G}^{1}$ to $G^{0}$,
called the \emph{source} function; and where $r$ is a function from
$\mathcal{G}^{1}$ to the power set of $G^{0}$,
%
%
%
%
$\mathcal{P}\left(  G^{0}\right)  $, such that $r(e)$ is non-empty for each
$e\in\mathcal{G}^{1}$. We write $\mathcal{G}^{0}$ for the smallest
subcollection of $\mathcal{P}\left(  G^{0}\right)  $ that contains $\left\{
v\right\}  $, for each $v\in G^{0}$ and contains $r\left(  e\right)  $ for all
$e\in\mathcal{G}^{1}$, and is closed under finite union and intersections. A
\textit{finite path} in $\mathcal{G}$ is either an element of $\mathcal{G}%
^{0}$ or a sequence of edges $e_{1}\ldots e_{k}$ in $\mathcal{G}^{1}$ where
$s\left(  e_{i+1}\right)  \in r\left(  e_{i}\right)  $ for $1\leq i\leq k$. If
we write $\alpha=e_{1}\ldots e_{k}$, the length $\left|  \alpha\right|  $ of
$\alpha$ is just $k$. The length $|A|$ of a path $A\in\mathcal{G}^{0}$ is
zero. We define $r\left(  \alpha\right)  =r\left(  e_{k}\right)  $ and
$s\left(  \alpha\right)  =s\left(  e_{1}\right)  $. For $A\in\mathcal{G}^{0}$,
we set $r\left(  A\right)  =A=s\left(  A\right)  $. The set of
finite paths in $\mathcal{G}$ is denoted by $\mathcal{G}^{\ast}$. The set of
infinite paths $\gamma=e_{1}e_{2}\ldots$ in $\mathcal{G}$ is denoted by
$\mathfrak
{p}^{\infty}$. The length $\left|  \gamma\right|  $ of $\gamma\in\mathfrak
{p}^{\infty}$ is defined to be $\infty$. A vertex $v$ in $\mathcal{G}$ is
called a \emph{sink} if $\left|  s^{-1}\left(  v\right)  \right|  =0$ and is
called an \emph{infinite emitter} if $\left|  s^{-1}\left(  v\right)  \right|
=\infty$. We say that a vertex $v$ is a \emph{singular vertex} if it is either
a sink or an infinite emitter. Finally given vertices $v,w\in G^{0}$, we write
$w\geq v$ to mean that there exists a path $\alpha\in\mathcal{G}^{\ast}$ with
$s\left(  \alpha\right)  =w$ and $v\in r\left(  \alpha\right)  $. Also we
write $G^{0}\geq\left\{  v\right\}  $ to mean that $w\geq v$, for all $w\in
G^{0}$. See \cite[p.8]{MTPS}.

\begin{blanket}
\label{blanket assumption}\emph{Throughout the paper we will assume that there
are no sinks in $\mathcal{G}$, unless otherwise specified.}
\end{blanket}

The reason for this is that we want to investigate $C^{\ast}\left(
\mathcal{G}\right)$ using an ultragraph groupoid $\mathfrak{G}_{\mathcal{G}%
}$, whose unit space $\mathfrak{G}_{\mathcal{G}}^{\left(  0\right)  }$
consists of paths that cannot end at a sink. So if we want to examine
$C^{\ast}\left(  \mathcal{G}\right)  $ from a groupoid perspective, then sinks must be excluded.

%
%
%
%
A semigroup $S$ is called an \emph{inverse semigroup} if for each $s\in S$,
there exists a unique element $t\in S$ such that $sts=s$ and $tst=t$. We write
the element $t$ as $s^{\ast}$. Note that $s^{\ast\ast}=s$. Every element
$ss^{\ast}$ belongs to the set $E\left(  S\right)  $ of \emph{idempotents} of
$S$. The set $E\left(  S\right)  $ is a commutative subsemigroup of $S$ and so
is a semilattice. There is a natural order on $E\left(  S\right)  $ given by
declaring $e\leq f$ if and only if $ef=e$, $e$ and $f\in E\left(  S\right)  $,
see \cite[Proposition 2.1.1, p.22]{APB}.

\section{{\protect\small THE INVERSE SEMIGROUP $S_{\mathcal{G}}$ OF AN
ULTRAGRAPH $\mathcal{G}$ AND THE UNIVERSAL GROUPOID FOR $S_{\mathcal{G}}$}}

In this section we have two main objectives. The first is to obtain an inverse
semigroup model $S_{\mathcal{G}}$ for an ultragraph $\mathcal{G}$. The second
is to identify the universal groupoid $H_{\mathcal{G}}$ for $S_{\mathcal{G}}$.
The definition of $S_{\mathcal{G}}$ stems from the representation theory of
$\mathcal{G}$. Recall the following definition due to Tomforde \cite{MTT}.

\begin{definition}
\label{cuntz-pimsner rep.} A \textit{representation} of $\mathcal{G}$ on a
Hilbert space $\mathcal{H}$ is given by a family $\left\{  p_{A}%
:A\in\mathcal{G}^{0}\right\}  $ of projections, and a family $\left\{
s_{e}:e\in\mathcal{G}^{1}\right\}  $ of partial isometries with mutually
orthogonal ranges such that:

\begin{enumerate}
\item [(i)]$p_{\emptyset}=0,$ $p_{A}p_{B}=p_{A\cap B}$,\ and \ $p_{A\cup
B}=p_{A}+p_{B}-p_{A\cap B}$, for all \ $A$, $B\in\mathcal{G}^{0}$;

\item[(ii)] $s_{e}^{\ast}s_{e}=p_{r\left(  e\right)  }$, for all
$e\in\mathcal{G}^{1}$;

\item[(iii)] $s_{e}s_{e}^{\ast}\leq p_{s\left(  e\right)  }$, for all
$e\in\mathcal{G}^{1}$;

\item[(iv)] $p_{v}=\sum_{s\left(  e\right)  =v}s_{e}s_{e}^{\ast}$, whenever
$0<\left|  s^{-1}\left(  v\right)  \right|  <\infty$.
\end{enumerate}
\end{definition}

The family $\left\{  s_{e},p_{A}:e\in\mathcal{G}^{1},A\in\mathcal{G}%
^{0}\right\}  $ is also called a \emph{Cuntz-Krieger\ }$\mathcal{G}$-family. For
a path $\alpha:=e_{1}\ldots e_{n}\in\mathcal{G}^{\ast}$ we define $s_{\alpha}$
to be $s_{e_{1}}\cdots s_{e_{n}}$ if $\left|  \alpha\right|  \geq1$ and
$p_{A}$ if $\alpha=A\in\mathcal{G}^{0}$.

Every inverse semigroup can be realized as $^{\ast}$-semigroup of partial
isometries on a Hilbert space. See \cite[Proposition 2.1.4]{APB}. If
$\mathcal{G}\mathbb{=}\left(  G^{0},\mathcal{G}^{1},r,s\right)  $ is an
ultragraph, and if $\left\{  s_{e},p_{A}:e\in\mathcal{G}^{1},A\in
\mathcal{G}^{0}\right\}  $ is a universal Cuntz-Krieger $\mathcal{G}$-family
realized on a Hilbert space $\mathcal{H}$, we know that the $C^{\ast}$-algebra
$C^{\ast}\left(  \mathcal{G}\right)  $ may be identified with the closed span,
$\overline{span}\{ s_{\alpha}p_{A}s_{\beta}^{\ast}:\alpha,\beta\in
\mathcal{G}^{\ast},A\in\mathcal{G}^{0}\}$, see \cite[p.7]{MTT}. Note that for
each $\alpha$, $\beta\in\mathcal{G}^{\ast}$, and $A\in\mathcal{G}^{0}$ with
$A\subseteq r\left(  \alpha\right)  \cap r\left(  \beta\right)  $, the
operator on $\mathcal{H}$, $T_{(\left(  \alpha,A),(\beta,A\right)
)}:=s_{\alpha}p_{A}s_{\beta}^{\ast}$ \ is a partial isometry, such that
$T_{(\left(  \alpha,A),(\beta,A\right)  )}^{\ast}=T_{(\left(  \beta
,A),(\alpha,A\right)  )}$. So we obtain a $^{\ast}$-semigroup of partial
isometries on $\mathcal{H}$, and therefore an inverse semigroup, which we
shall denote by $S_{\mathcal{G}}$. Using properties of the generators for
$C^{\ast}\left(  \mathcal{G}\right)  $, $s_{\alpha}p_{A}s_{\beta}^{\ast}$,
$\alpha$, $\beta\in\mathcal{G}^{\ast}$, and $A\in\mathcal{G}^{0}$ (see
\cite[Lemma 2.8 and Lemma 2.9]{MTT}), we will describe the inverse semigroup
$S_{\mathcal{G}}$ for $\mathcal{G}$ in a fashion that is independent of any
Hilbert space representation. It turns out that our inverse semigroup is
analogous to the inverse semigroup model for an ordinary graph that Paterson
obtains in \cite{APP}.

\begin{remark}
\label{principal difficulty}While it is clear in a broad sense what one must
do to follow the path laid out by Paterson, there is an important difficulty
that must be surmounted. It may be helpful, therefore, to call attention to it
here for the purpose of motivating later detail. Condition (i) in Definition
\ref{cuntz-pimsner rep.} is the source of the difficulty. Notice that it says
that the family $\left\{  p_{A}:A\in\mathcal{G}^{0}\right\}  $ is a
{}``proto-spectral measure\textquotedblright\ defined on the {}%
``lattice\textquotedblright\ $\mathcal{G}^{0}$. As we build our inverse
semigroup and groupoid models, we will have to keep track of how to guarantee
condition (i) in what we are doing.
\end{remark}

For this purpose, we find it helpful to {}``enrich\textquotedblright\ the path
notation discussed in subsection \ref{Notation} and introduce the notion of
what we like to call {}``ultrapaths\textquotedblright. For $n\geq1,$ we define
$\mathfrak{p}^{n}:=\{\left(  \alpha,A\right)  :\alpha\in\mathcal{G}^{\ast
},\left\vert \alpha\right\vert =n,$ $A\in\mathcal{G}^{0},A\subseteq r\left(
\alpha\right)  \}$. We specify that $\left(  \alpha,A\right)  =(\beta,B)$ if
and only if $\alpha=\beta$ and $A=B$. We set $\mathfrak{p}^{0}:=\mathcal{G}%
^{0}$ and we let $\mathfrak{p}:=\coprod\limits_{n\geq0}\mathfrak{p}^{n}$. \ We
define the length of a pair $\left(  \alpha,A\right)  $, $\left\vert \left(
\alpha,A\right)  \right\vert $ to be the length of $\alpha$, $\left\vert
\alpha\right\vert $. We call $\mathfrak{p}$ the \emph{ultrapath space}
associated with $\mathcal{G}$ and the elements of $\mathfrak{p}$ are called
\emph{ultrapaths}. We may extend the range map $r$ and the source map $s$ to
$\mathfrak{p}$ by the formulas, $r\left(  \left(  \alpha,A\right)  \right)
=A$ and $s\left(  \left(  \alpha,A\right)  \right)  =s\left(  \alpha\right)
$. Each $A\in\mathcal{G}^{0}$ is regarded as an ultrapath of length zero and
we define $r\left(  A\right)  =s\left(  A\right)  =A$. It will be convenient
to embed $\mathcal{G}^{\ast}$ in $\mathfrak{p}$ by sending $\alpha$ to
$(\alpha,r(\alpha))$, if $|\alpha|\geq1$, and by sending $A$ to $A$ for all
$A\in\mathcal{G}^{0}$ (See \cite{MTT}.) In a sense, our formation of the
ultrapath space $\mathfrak{p}$ is analogous to the process of forming the
\emph{disjoint union} of a family of not-necessarily-disjoint sets, i.e.,
their co-product.

\begin{notation}
\label{UltrapathAlgebra}Generic elements of $\mathfrak{p}$ will be denoted by
lower case letters at the end of the alphabet: $x$, $y$, $z$ and $w$. However
if $x\in\mathfrak{p}$, and if $x\in\mathfrak{p}^{0}:=\mathcal{G}^{0}$, we
think of $x$ as a set $A$ in $\mathcal{G}^{0}$. Otherwise we think of $x$ as a
pair, say $\left(  \alpha,A\right)  $, with $\left|  \alpha\right|  \geq1$.
Algebraically, we treat $\mathfrak{p}$ like a small category and say that a
product $x\cdot y$ is defined only when $r\left(  x\right)  \cap s\left(
y\right)  \neq\emptyset$\footnote{The notation is a little inconsistant. By
definition, $s(y)$ is a set only when $y\in\mathfrak{p}^{0}$; otherwise,
$s(y)$ is a point in $G^{0}$. In the latter case, we really want to identify
$s(y)$ with $\{s(y)\}$. We shall do this whenever it is convenient and not add
extra notation to distinguish between $s(y)$ and $\{s(y)\}$.}. When $x\cdot y$
is defined, the product is effectively concatenation of $x$ and $y$. That is,
if $x=(\alpha,A)$ and $y=(\beta,B)$, then $x\cdot y$ is defined if and only if
$s(\beta)\in A$, and in this case, $x\cdot y:=(\alpha\beta,B)$. Also we
specify that:
\begin{equation}
x\cdot y=\left\{
\begin{array}
[c]{ll}%
x\cap y & \text{if }x,y\in\mathcal{G}^{0}\text{ and if }x\cap y\neq\emptyset\\
y & \text{if }x\in\mathcal{G}^{0}\text{, }\left|  y\right|  \geq1\text{, and
if }x\cap s\left(  y\right)  \neq\emptyset\\
x_{y} & \text{if }y\in\mathcal{G}^{0}\text{, }\left|  x\right|  \geq1\text{,
and if }r\left(  x\right)  \cap y\neq\emptyset
\end{array}
\right.  \label{specify}%
\end{equation}
where, if $x=\left(  \alpha,A\right)  $, $\left|  \alpha\right|  \geq1$ and if
$y\in\mathcal{G}^{0}$, the expression $x_{y}$ is defined to be $\left(
\alpha,A\cap y\right)  $. Observe also that the range of $x_{y}$, $r\left(
x_{y}\right)  $, becomes $r\left(  x_{y}\right)  =r\left(  x\right)  \cap y$.
Given $x,y\in\mathfrak{p}$, we say that $x$ has $y$ as an initial segment if
$x=y\cdot x^{\prime}$, for some $x^{\prime}\in\mathfrak{p}$, with $s\left(
x^{\prime}\right)  \cap r\left(  y\right)  \neq\emptyset$. We shall say that
$x$ and $y$ in $\mathfrak{p}$ are comparable if $x$ has $y$ as an initial
segment or vice versa. Furthermore, the equation, $x=x\cdot y$, holds if and
only if $y\in\mathcal{G}^{0}$ and $r\left(  x\right)  \subseteq y$.
\end{notation}

Recall from subsection \ref{Notation} that $\mathfrak{p}^{\infty}$ denotes the
set of all infinite paths. We extend the source map $s$ to $\mathfrak
{p}^{\infty}$, by defining $s(\gamma)=s\left(  e_{1}\right)  $, where
$\gamma=e_{1}e_{2}\ldots$. We may concatenate pairs in $\mathfrak{p}$, with
infinite paths in $\mathfrak{p}^{\infty}$ as follows. If $y=\left(
\alpha,A\right)  \in\mathfrak{p}$, and if $\gamma=e_{1}e_{2}\ldots\in
\mathfrak{p}^{\infty}$ are such that $s\left(  \gamma\right)  \in r\left(
y\right)  =A$, then the expression $y\cdot\gamma$ is defined to be
$\alpha\gamma=\alpha e_{1}e_{2}...\in\mathfrak{p}^{\infty}$. If $y=$
$A\in\mathcal{G}^{0}$, we define $y\cdot\gamma=A\cdot\gamma=\gamma$ whenever
$s\left(  \gamma\right)  \in A$. Of course $y\cdot\gamma$ is not defined if
$s\left(  \gamma\right)  \notin r\left(  y\right)  =A$. In this way, we get an
{}``action\textquotedblright\ of $\mathfrak{p}$ on $\mathfrak{p}^{\infty}$.

\begin{definition}
\label{inverse semigroup}Let $S_{\mathcal{G}}:=\{\left(  x,y\right)
:x,y\in\mathfrak{p},r\left(  x\right)  =r\left(  y\right)  \}\cup\left\{
\omega\right\}  $. We define an involution on $S_{\mathcal{G}}$ by
$\omega^{\ast}=\omega$ and $\left(  x,y\right)  ^{\ast}=\left(  y,x\right)  $,
and we define a product on $S_{\mathcal{G}}$ by the following requirements:
\end{definition}

\begin{enumerate}
\item $\left(  x,r\left(  x\right)  \right)  \left(  r\left(  y\right)
,y\right)  :=\left(  x\cdot r\left(  y\right)  ,y\cdot r\left(  x\right)
\right)  $, for all $x$ and $y\in\mathfrak{p}$ with $r\left(  x\right)  \cap
r\left(  y\right)  \neq\emptyset$.

\item  If $x$ has $z$ as an initial segment, so $x=z\cdot x^{\prime}$ for some
$x^{\prime}\in\mathfrak{p}$, then $\left(  w,z\right)  \left(  x,y\right)
=\left(  w,z\right)  \left(  z\cdot x^{\prime},y\right)  :=\left(  w\cdot
x^{\prime},y\right)  $.

\item  If $z$ has $x$ as an initial segment, so $z=x\cdot z^{\prime}$ for some
$z^{\prime}\in\mathfrak{p}$, then $\left(  w,z\right)  \left(  x,y\right)
=\left(  w,x\cdot z^{\prime}\right)  \left(  x,y\right)  :=\left(  w,y\cdot
z^{\prime}\right)  $.

\item  All other products are defined to be $\omega$.
\end{enumerate}

\begin{proposition}
The set $S_{\mathcal{G}}$ with this involution and product is an inverse semigroup.
\end{proposition}

\begin{proof}
We first have to show that the product in $S_{\mathcal{G}}$ is associative. If
one of the terms is $\omega$, then this is obvious. So let $\left(
x_{i},y_{i}\right)  \in S_{\mathcal{G}}$ \ ($1\leq i\leq3$), and set
$s_{1}=\left(  x_{1},y_{1}\right)  \left(  \left(  x_{2},y_{2}\right)  \left(
x_{3},y_{3}\right)  \right)  $ and $s_{2}=(\left(  x_{1},y_{1}\right)  \left(
x_{2},y_{2}\right)  )\left(  x_{3},y_{3}\right)  $. We have to show that
$s_{1}=s_{2}$.

The cases that give $s_{1}\neq\omega$ are the following (for appropriate
ultrapaths $z,w$):

\begin{enumerate}
\item $y_{2}=x_{3}\cdot z$, for some $z\in\mathfrak{p}$ and $y_{1}=x_{2}\cdot
w$ for some $w\in\mathfrak{p}$;

\item $y_{2}=x_{3}$, and $y_{1}=x_{2}\cdot w$ for some $w\in\mathfrak{p}$;

\item $y_{2}=x_{3}\cdot z$, for some $z\in\mathfrak{p}$ and $y_{1}=x_{2}$;

\item $y_{2}=x_{3}\cdot z$, for some $z\in\mathfrak{p}$ and $x_{2}=y_{1}\cdot
w$ for some $w\in\mathfrak{p}$;

\item $x_{i}\in\mathcal{G}^{0}$, for some $i$, with $s\left(  x_{i}\right)
\cap r\left(  y_{i-1}\right)  \neq\emptyset$, together with the cases above;

\item $y_{i}\in\mathcal{G}^{0}$, for some $i$, with $r\left(  x_{i+1}\right)
\cap s\left(  y_{i}\right)  \neq\emptyset$, together with the cases above.
\end{enumerate}

One checks directly that in each case, $s_{2}=s_{1}$. ( In case 4, one needs
to consider separately the cases $x_{2}=y_{1}\cdot z^{\prime}$ and
$y_{1}=x_{2}\cdot z^{\prime}$ ). So if $s_{1}\neq\omega$, then $s_{1}=s_{2}$.
Similarly, one shows that if $s_{2}\neq\omega$, then $s_{2}=s_{1}$. The
associative law then follows.

Next we have to show that for each $s\in S_{\mathcal{G}}$, $s^{\ast}$ is the
only $s^{\prime}$ for which $ss^{\prime}s=s$ and $s^{\prime}ss^{\prime
}=s^{\prime}$. If $s=\omega$, then this is trivial. So let $s=\left(
x,y\right)  \neq\omega$. If $s^{\prime}$ is such that $ss^{\prime}s=s$, then
$s^{\prime}=\left(  w,z\right)  $ for some ultrapaths $w$, $z$, with $z$, $x$
and $y$, $w$ comparable. (See Notation \ref{Notation}) Suppose first $z$ has
$x$ as an initial segment. We shall show that $z=x$. A similar argument will
show that $x=z$ if $x$ has $z$ as an initial segment. So if $z$ has $x$ as an
initial segment then $z=x\cdot z^{\prime}$ for some $z^{\prime}\in\mathfrak
{p}$ with $s\left(  z^{\prime}\right)  \cap r\left(  x\right)  \neq\emptyset$.
Then from the equation $ss^{\prime}s=s$ and from the definition of the product
on $S_{\mathcal{G}}$, the equality, $\left(  x,y\right)  =\left(  x,y\right)
\left(  w,y\cdot z^{\prime}\right)  $, holds. We see then that the product,
$\left(  x,y\right)  \left(  w,y\cdot z^{\prime}\right)  $, is not $\omega$.
So we have to consider two cases.

\begin{enumerate}
\item [Case I ]$w$ has $y$ as an initial segment. So $w=y\cdot w^{\prime}$ for
some $w^{\prime}\in\mathfrak{p}$ with $s\left(  w^{\prime}\right)  \cap
r\left(  y\right)  \neq\emptyset$. Then we have $\left(  x,y\right)  =\left(
x,y\right)  \left(  y\cdot w^{\prime},y\cdot z^{\prime}\right)  =\left(
x\cdot w^{\prime},y\cdot z^{\prime}\right)  $, by definition. (See 2
Definition \ref{inverse semigroup}) Thus the equation, $\left(  x\cdot
w^{\prime},y\cdot z^{\prime}\right)  =\left(  x,y\right)  $, holds, which
implies that the equation, $y\cdot z^{\prime}=y$, holds as well. Therefore
$z^{\prime}$ belongs to $\mathcal{G}^{0}$, and $r\left(  x\right)  =r\left(
y\right)  \subseteq z^{\prime}$. Then $x=x\cdot z^{\prime}=z$.

\item[Case II] $y$ has $w$ as an initial segment. In this case a similar proof
gives us that $x=z$ as well.
\end{enumerate}

Similarly, by considering the equation $s^{\prime}ss^{\prime}=s^{\prime}$, one
can show that $w=y$, in the situation when $y$ and $w$ are comparable. Then
$s^{\prime}=s^{\ast}$. Clearly $ss^{\ast}s=s$ and $s^{\ast}ss^{\ast}=s^{\ast}$.
\end{proof}

The next theorem shows that there is a bijection between the class of
representations of the ultragraph $\mathcal{G}$, and \emph{certain} class of
representations of the inverse semigroup $S_{\mathcal{G}}$. Compare with
\cite[Theorem 2, (a), (b)]{APP}, and of course please keep in mind Remark
\ref{principal difficulty}.

\begin{theorem}
\label{nat.correspondence} There is a natural one-to-one correspondence
between :

\begin{enumerate}
\item [(a)]the class $R_{\mathcal{G}}$ of representations of $\mathcal{G}$; and

\item[(b)] the class $R_{S_{\mathcal{G}}}$ of representations $\pi$ of
$S_{\mathcal{G}}$ such that:

\begin{enumerate}
\item [(i)]$\pi(\omega)=0,$

\item[(ii)] $\pi\left(  v,v\right)  -\sum_{s\left(  e\right)  =v}\pi\left(
e,e\right)  =0$, for every $v\in G^{0}$, with $0<\left|  s^{-1}\left(
v\right)  \right|  <\infty$, and

\item[(iii)] $\pi\left(  A\cup B,A\cup B\right)  -\pi\left(  A,A\right)
-\pi\left(  B,B\right)  +\pi\left(  A\cap B,A\cap B\right)  =0$, for every
$A$, $B\in\mathcal{G}^{0}$.
\end{enumerate}
\end{enumerate}
\end{theorem}

\begin{proof}
Let $\left\{  p_{A},s_{e}:A\in\mathcal{G}^{0},e\in\mathcal{G}^{1}\right\}  $
be a representation of $\mathcal{G}$, realized on a Hilbert Space
$\mathcal{H}$. Define a $\ast$-map $\pi:S_{\mathcal{G}}\rightarrow B\left(
\mathcal{H}\right)  $ by $\pi\left(  x,y\right)  =S_{x}S_{y}^{\ast}$, and
$\pi\left(  \omega\right)  =0$, where $S_{x}$ is defined to be $s_{\alpha
}p_{A}$ if $x=\left(  \alpha,A\right)  $, and $p_{A}$ if $x=A\in
\mathcal{G}^{0}$.\ Then $\pi$ is a $\ast$-homomorphism. The proof is simple.
One checks that $\pi\left(  st\right)  =\pi\left(  s\right)  \pi\left(
t\right)  $ for the different kinds of product using properties of the
generator $s_{\alpha}p_{A}s_{\beta}^{\ast}$ for $C^{\ast}\left(
\mathcal{G}\right)  $. See \cite[Lemma 2.8 and Lemma 2.9]{MTT}. For example
checking carefully all the cases, we have $\pi\left[  \left(  w,z\right)
\right]  \pi\left[  \left(  z\cdot x^{\prime},y\right)  \right]  =(S_{w}%
S_{z}^{\ast})(S_{z\cdot x^{\prime}}S_{y}^{\ast})=S_{w\cdot x^{\prime}}%
S_{y}^{\ast}$=$\pi\left[  \left(  w\cdot x^{\prime},y\right)  \right]
=\pi\left[  \left(  w,z\right)  \left(  z\cdot x^{\prime},y\right)  \right]
$. Since $\pi\left(  \omega\right)  =0$ and $s_{e}s_{e}^{\ast}=\pi\left(
e,e\right)  $, it follows from (i) and (iv) of Definition \ref{cuntz-pimsner
rep.} that $\pi\in R_{S_{\mathcal{G}}}$.

Conversely, any $\pi\in R_{S_{\mathcal{G}}}$ determines an element of
$R_{\mathcal{G}}$ by taking $p_{A}=:\pi\left(  A,A\right)  $, if
$A\neq\emptyset$ and $0$ otherwise. For $e\in\mathcal{G}^{1}$ take $s_{e}%
:=\pi\left(  e,r\left(  e\right)  \right)  $. Since $p_{\emptyset}=0$ and
$p_{A}p_{B}=\pi\left(  A,A\right)  \pi\left(  B,B\right)  =\pi\left(  A\cap
B,A\cap B\right)  =p_{A\cap B}$, (i) of Definition \ref{cuntz-pimsner rep.}
follows. Also since $s_{e}s_{e}^{\ast}=\pi\left(  e,r\left(  e\right)
\right)  \pi\left(  r\left(  e\right)  ,e\right)  =\pi\left(  e,e\right)  $,
(iv) of Definition \ref{cuntz-pimsner rep.} follows as well. (i) and (iii)
follow since $\pi$ is a $\ast$-homomorphism on $S_{\mathcal{G}}$. This
establishes the correspondence between the classes of representations of (a)
and (b).
\end{proof}

Every representation $\pi$ of $S_{\mathcal{G}}$ by partial isometries on
Hilbert space gives a bounded representation $\pi$ of $l^{1}\left(
S_{\mathcal{G}}\right)  $ in a natural way, and $C^{\ast}\left(
S_{\mathcal{G}}\right)  $ is just the enveloping $C^{\ast}$-algebra of
$l^{1}\left(  S_{\mathcal{G}}\right)  $ obtained by taking the biggest norm
coming from all such $\pi$'s. The $C^{\ast}$-algebra that we want here, which
we will denote by $C_{0}^{\ast}\left(  S_{\mathcal{G}}\right)  $, is obtained
in the same way but using only $\pi$'s for which $\pi\left(  \omega\right)
=0=\pi(\left(  v,v\right)  -\sum_{s\left(  e\right)  =v}\left(  e,e\right)
)$, for every $v\in G^{0}$, with $0<\left\vert s^{-1}\left(  v\right)
\right\vert <\infty$ and for which $\pi(\left(  A\cup B,A\cup B\right)
-\left(  A,A\right)  -\left(  B,B\right)  +\left(  A\cap B,A\cap B\right)  )$
$=$ $0$, for every $A$ and $B\in\mathcal{G}^{0}$. In fact $C_{0}^{\ast}\left(
S_{\mathcal{G}}\right)  $ is the quotient $C^{\ast}$-algebra $C^{\ast
}(S_{\mathcal{G}})/I$, where $I$ is the closed ideal of $C^{\ast}(l^{1}\left(
S_{\mathcal{G}}\right)  )$ generated by elements of the form: $\omega$,
$\left(  v,v\right)  -\sum_{s\left(  e\right)  =v}\left(  e,e\right)  $, for
every $v\in G^{0}$, with $0<\left\vert s^{-1}\left(  v\right)  \right\vert
<\infty$ and $\left(  A\cup B,A\cup B\right)  -\left(  A,A\right)  -\left(
B,B\right)  +\left(  A\cap B,A\cap B\right)  $, for every $A$ and
$B\in\mathcal{G}^{0}$. \emph{A priori} the quotient $C_{0}^{\ast
}\left(  S_{\mathcal{G}}\right)  $ could be zero, but Tomforde shows that it
isn't, in \cite{MTT}, and, of course, our analysis will show this, too.

The next objective is to identify the universal groupoid of $S_{\mathcal{G}}$
for a general ultragraph $\mathcal{G}$. The universal groupoid (\cite[Ch.
4]{APB}) $H$ of a countable inverse semigroup $S$ is constructed as follows.
The unit space $H^{\left(  0\right)  }$ of $H$ is the set of non-zero
semicharacters, i.e., homomorphisms, $\chi$, from the commutative inverse
subsemigroup of idempotents $E\left(  S\right)  $ in $S$ to the semigroup
$\{0,1\}$ (under multiplication). The topology on $H^{\left(  0\right)  }$ is
the topology of pointwise convergence on $E\left(  S\right)  $. This implies
that the family of sets $D_{e,e_{1,\ldots},e_{n}}=D_{e}\cap D_{e_{1}}^{c}%
\cap...\cap D_{e_{n}}^{c}$, (with $c$ standing for {}%
``complement\textquotedblright\ and $e,e_{i}\in E\left(  S\right)  $, $e\geq
e_{i}$, $1\leq i\leq n$) is a basis for the topology of $H^{\left(  0\right)
}$, \cite[Chap.4, p.174]{APB}. With respect to this topology, the space
$H^{\left(  0\right)  }$ is locally compact, totally disconnected and
Hausdorff, \cite[p.173]{APB}. There is a natural right action of $S$ on
$H^{\left(  0\right)  }$ given as follows. First, an element $\chi\in
H^{\left(  0\right)  }$ is in the domain $D_{s}$ of $s\in S$ if $\chi\left(
ss^{\ast}\right)  =1$. The element $\chi\cdot s\in H^{\left(  0\right)  }$ is
then defined by the equation $(\chi\cdot s)\left(  e\right)  =\chi\left(
ses^{\ast}\right)  $, for $e\in E\left(  S\right)  $. The map $\chi
\rightarrow\chi\cdot s$ is a homeomorphism from $D_{s}$ onto $D_{s^{\ast}}$.
Theorem 4.3.1 of \cite{APB} shows that the universal groupoid $H$ for $S$ is
the quotient
\[
\{\left(  \chi,s\right)  :\chi\in D_{s},s\in S\}/\backsim
\]
where, by definition, $\left(  \chi,s\right)  \backsim\left(  \chi^{\prime
},t\right)  $ whenever $\chi=\chi^{\prime}$ and there exists $e\in E\left(
S\right)  $ such that $\chi\left(  ss^{\ast}\right)  =\chi\left(  tt^{\ast
}\right)  $ and $es=et$. That is, two pairs $\left(  \chi,s\right)  $ and
$\left(  \chi^{\prime},t\right)  $ are equivalent if and only if $\chi
=\chi^{\prime}$ and $s$ and $t$ have the same germ at $\chi$. The composable
pairs are pairs of the form $\left(  \overline{\left(  \chi,s\right)
},\overline{\left(  \chi\cdot s,t\right)  }\right)  $, where $\chi\in D_{s}$,
$\chi\cdot s\in D_{t}$, and $s,t\in S$; and the product and inversion on $H$
are given by the maps $\left(  \overline{\left(  \chi,s\right)  }%
,\overline{\left(  \chi\cdot s,t\right)  }\right)  \rightarrow\overline
{\left(  \chi,st\right)  }$ and $\overline{\left(  \chi,s\right)  }%
\rightarrow\overline{\left(  \chi\cdot s,s^{\ast}\right)  }$, respectively.
Also $H$ is an $r$-discrete groupoid, where the topology on $H$ is the germ
topology. It has a basis consisting of sets of the form $D(U,s)$, where $s\in
S$, $U$ is an open subset of $D_{s}$, and $D(U,s):=\{\overline{(\chi,s)}%
:\chi\in U\}$. Further, the map $\Psi$, where $\Psi\left(  s\right)
=\{\overline{\left(  \chi,s\right)  }:\chi\in D_{s}\}$ is an inverse semigroup
isomorphism from $S$ into the ample semigroup $H^{a}$. (For any $r$-discrete
groupoid $G$, the \emph{ample semigroup} $G^{a}$ is the inverse semigroup of
compact open, Hausdorff $G$-sets in $G$, see \cite[Chap.2, Definition 2.2.4,
Proposition 2.2.6]{APB}).

The description of the universal groupoid $H=H_{\mathcal{G}}$ of
$S=S_{\mathcal{G}}$ is in many respects similar to the description of the
universal groupoid of the graph inverse semigroup associated to a graph.
However, there are some important differences. To highlight them, we follow as
closely as possible the discussion for the directed graph inverse semigroup
obtained by Paterson in \cite{APP}. The key is to identify the unit space
$H^{\left(  0\right)  }$ of $H$. The semigroup of idempotents of
$S_{\mathcal{G}}$, which we denote by $E\left(  S_{\mathcal{G}}\right)  $, is
the set $\{\left(  x,x\right)  :x\in\mathfrak{p}\}\cup\left\{  \omega\right\}
$. Recall that for any inverse semigroup $S$, there is a natural order on the
idempotent subsemigroup $E(S)$ defined by the prescription $e\leq f$ \ if and
only if $ef=e$, $(e,$ $f\in E\left(  S\right)  )$. In our setting, the order
on $E\left(  S_{\mathcal{G}}\right)  $ may be described in terms of path
length and set inclusion, as the following remark indicates. We leave the
proof to the reader.

\begin{remark}
\label{order in E}If the product in $E\left(  S_{\mathcal{G}}\right)  $,
$\left(  x,x\right)  \left(  z,z\right)  $, is not $\omega$ then the
inequality, $\left(  z,z\right)  \leq\left(  x,x\right)  $, holds if and only
if either $|z|>|x|$ or, if $|z|=|x|$, then $r\left(  z\right)  \subseteq r\left(
x\right)  $.
\end{remark}

As is the case with any idempotent semigroup, the elements in $E\left(
S_{\mathcal{G}}\right)  $ can themselves be regarded as semicharacters on
$E(S_{\mathcal{G}})$, and for each element, there is a filter associated with
it. That is, given $e\in E(S_{\mathcal{G}})$, then $\chi_{e}$ is the
semicharacter of $E\left(  S_{\mathcal{G}}\right)  $ defined by: $\chi
_{e}(f)=1$ if $f\geq e$ and is $0$ otherwise. Its filter $\tilde{e}$, is the
set of idempotents $\tilde{e}:=\{f\in E(S_{\mathcal{G}}):f\geq e\}$, see
\cite[p.173-174]{APB}. That is, $\tilde{e}$ is the principal filter determined
by $e$. Furthermore, the set $\widetilde{E\left(  S_{\mathcal{G}}\right)
}:=\{\tilde{e}:e\in E\left(  S_{\mathcal{G}}\right)  \}$ is dense in the set
of all nonzero semicharacters of $E\left(  S_{\mathcal{G}}\right)  $, which we
shall denote by $\widehat{E\left(  S_{\mathcal{G}}\right)  }$. (See \cite[
Proposition 4.3.1 p.174]{APB}\footnote{It is customary to denote the principal
filter determined by idempotent $e$ by $\overline{e}$, and then we might write
$\overline{E\left(  S\right)  }$ for the collection of all such filters.
However, in our setting, this leads to awkward statements like {}%
``$\overline{E\left(  S\right)  }$ is dense in $X$'', which in turn would lead one to
believe $\overline{E\left(  S\right)  }=X$.}.) The collection of \ {}``subsets
of generalized vertices''\ in the ultragraph $\mathcal{G}$, which, recall, is
denoted $\mathcal{G}^{0}$ and is an idempotent inverse semigroup in its own right under
intersection, may be viewed as an sub-inverse-semigroup of $E\left(
S_{\mathcal{G}}\right)  $ via the map $A\rightarrow(A,A)$. Then every
semicharacter on $E(S_{\mathcal{G}})$ restricts to one on $\mathcal{G}^{0}$.
This leads to the inclusion, $H_{\mathcal{G}}^{\left(  0\right)  }%
=\widehat{E\left(  S_{\mathcal{G}}\right)  }\subseteq\widehat{\mathcal{G}^{0}%
}$, where $\widehat{\mathcal{G}^{0}}$ denotes the set of all non-zero
semicharacters of $\mathcal{G}^{0}$. The topology on $\widehat{\mathcal{G}%
^{0}}$ is the topology of pointwise convergence on $\mathcal{G}^{0}$.
Consequently, the family of compact open sets, $D\left(  A,A\right)
:=\{\chi\in\widehat{\mathcal{G}^{0}}:\chi\left(  A\right)  =1\}$,
$A\in\mathcal{G}^{0}$, is a subbasis. It follows that the space,
$\widehat{\mathcal{G}^{0}}$, is locally compact and Hausdorff as well. (See
\cite[Chapter 4, p.174]{APB}.) We want to emphasize here that the space
$\widehat{\mathcal{G}^{0}}$ is \emph{huge}. Since it contains the discrete
space $G^{0}$, $\widehat{\mathcal{G}^{0}}$ contains the Stone-\v{C}ech
compactification $\beta G^{0}$ of $G^{0}$. Since $\mathcal{G}^{0}$ is an
idempotent inverse semigroup, the elements in $\mathcal{G}^{0}$ can themselves
be regarded as semicharacters of $\mathcal{G}^{0}$, and each element
determines a principal filter. That is, given $A\in\mathcal{G}^{0}$, then
$\chi_{A}$ is the semicharacter of $\mathcal{G}^{0}$ defined by: $\chi
_{A}(B)=1$ if $A\subseteq B$ and is $0$ otherwise. Its filter $\tilde{A}$, is
the set given then by: $\tilde{A}=\{B\in\mathcal{G}^{0}:A\subseteq B\}$.
Consequently, the set of all semicharacters of the form, $\chi_{A}$,
$A\in\mathcal{G}^{0}$, which we shall denote by, $\widetilde{\mathcal{G}^{0}}%
$, is dense in $\widehat{\mathcal{G}^{0}}$ \cite[ Proposition 4.3.1
p.174]{APB}. This fact plays an important
role in our efforts to overcome the difficulties alluded to in Remark
\ref{principal difficulty}.

Now we proceed to identify the unit space $H_{\mathcal{G}}^{\left(  0\right)
}=\widehat{E\left(  S_{\mathcal{G}}\right)  }$, of the universal groupoid
$H=H_{\mathcal{G}}$ of $S=S_{\mathcal{G}}$ and its topology in more concrete
terms. It is more convenient to discuss the space $\widehat
{E\left(  S_{\mathcal{G}}\right)  }$ in terms of filters rather than in terms of
 semicharacters.
Recall that given a nonzero semicharater $\chi$ in $\widehat{E\left(
S_{\mathcal{G}}\right)  }$ its filter, $\mathcal{A}_{\chi}$, is the set
$\mathcal{A}_{\chi}:=\left\{  \left(  x,x\right)  \in E\left(  S_{\mathcal{G}%
}\right)  :\chi\left(  x,x\right)  =1\right\}  $. (See \cite[p.173-174]{APB}.)
Each ultrapath, $y\in\mathfrak{p}$, defines a semicharacter of $E\left(
S_{\mathcal{G}}\right)  $ via the equation,
\begin{equation}
y\left(  x,x\right)  =\left\{
\begin{array}
[c]{cc}%
1 & \text{if }\left(  x,x\right)  \left(  y,y\right)  \neq\omega\text{,
}|x|<|y|\text{ or; if }|x|=|y|\text{, then }r\left(  x\right)  \supseteq
r\left(  y\right) \\
0 & \text{otherwise.}%
\end{array}
\right.  \label{ultras}%
\end{equation}
Also an infinite path $\gamma\in\mathfrak{p}^{\infty}$ defines a semicharacter
of $E\left(  S_{\mathcal{G}}\right)  $ via the equation:
\begin{equation}
\gamma\left(  x,x\right)  =\left\{
\begin{array}
[c]{cc}%
1 & \text{if }\gamma=x\cdot\gamma^{\prime},\gamma^{\prime}\in\mathfrak
{p}^{\infty},s\left(  \gamma^{\prime}\right)  \in r\left(  x\right) \\
0 & \text{otherwise.}%
\end{array}
\right.  \label{infinitep}%
\end{equation}
Consequently, the filter in $E\left(  S_{\mathcal{G}}\right)  $ determined by
the ultrapath $y$, $\mathcal{A}_{y}$, is the set, $\mathcal{A}_{y}:=\{\left(
x,x\right)  \in E\left(  S_{\mathcal{G}}\right)  :\left(  x,x\right)  \left(
y,y\right)  \neq\omega$ ; $|x|<|y|$ or, if $|x|=|y|$, then $r\left(  x\right)
\supseteq r\left(  y\right)  $ $\}$, while the filter in $E\left(
S_{\mathcal{G}}\right)  $ determined by an infinite path $\gamma$,
$\mathcal{A}_{\gamma}$, is the set $\mathcal{A}_{\gamma}:=\{\left(
x,x\right)  \in E\left(  S_{\mathcal{G}}\right)  :\gamma=x\cdot\gamma^{\prime
},\gamma^{\prime}\in\mathfrak{p}^{\infty},s\left(  \gamma^{\prime}\right)  \in
r\left(  x\right)  \}$. 


\begin{remark}
\label{equality of filters} The correspondence between ultrapaths and filters
is one-to-one, that is for ultrapaths $y$ and $z$ in $\mathfrak{p}$,
$\mathcal{A}_{y}=\mathcal{A}_{z}$ if and only if $y=z$. Likewise for infinite
paths $\gamma$ and $\gamma^{\prime}$ in $\mathfrak{p}^{\infty}$,
$\mathcal{A}_{\gamma}=\mathcal{A}_{\gamma^{\prime}}$ if and only if
$\gamma=\gamma^{\prime}$. Furthermore, if $y$ is an ultrapath and if $\gamma$
is an infinite path, then $\mathcal{A}_{y}\neq\mathcal{A}_{\gamma}$.
\end{remark}

\begin{proof}
Recall that if $y$ is an ultrapath the set $\mathcal{A}_{y}$ is given by
$\mathcal{A}_{y}=\{\left(  x,x\right)  \in E\left(  S_{\mathcal{G}}\right)
:\left(  x,x\right)  \left(  y,y\right)  \neq\omega$ ; $|x|<|y|$ or, if
$|x|=|y|$, then $r\left(  x\right)  \supseteq r\left(  y\right)  \}$.
Obviously, if $y=z$ then $\mathcal{A}_{y}=\mathcal{A}_{z}$. So suppose that
the equality, $\mathcal{A}_{y}=\mathcal{A}_{z}$, holds. Then $\left(
y,y\right)  \in\mathcal{A}_{z}$ and $\left(  z,z\right)  \in\mathcal{A}_{y}$.
Hence the product in $E\left(  S_{\mathcal{G}}\right)  $, $\left(  y,y\right)
\left(  z,z\right)  $, is not $\omega$. Furthermore, since $\left(
y,y\right)  \in\mathcal{A}_{z}$, we see that $|y|\leq|z|$, and since
$\left(  z,z\right)  \in\mathcal{A}_{y}$, it follows that $\left|  y\right|
\geq\left|  z\right|  $. Hence $|z|=|y|$. But again, since $\left(
y,y\right)  \in\mathcal{A}_{z}$ it follows that $r\left(  z\right)  \subseteq
r\left(  y\right)  $ and since $\left(  z,z\right)  \in\mathcal{A}_{y}$ it
follows that then $r\left(  y\right)  \subseteq r\left(  z\right)  $.
Consequently, the equality, $r\left(  y\right)  =r\left(  z\right)  $, also
holds. To see that $y=z$, suppose first that one of the ultrapaths, $y$ or
$z$, is in $\mathcal{G}^{0}$. Then since $\left|  y\right|  =\left|  z\right|
$, it follows that the other is also in $\mathcal{G}^{0}$. But since $r\left(
y\right)  =r\left(  z\right)  $, it follows that $y=z$. Now suppose that $y$
and $z$ have positive length and recall that the inequality, $\left(
y,y\right)  \left(  z,z\right)  \neq\omega$, means that the paths $y$ and $z$
are comparable. (See Definition \ref{inverse semigroup} and see the paragraph following
 equation (\ref{specify}).) Suppose that $y$ has $z$ as an initial
segment. That is, suppose that $y=z\cdot y^{\prime}$ for some $y^{\prime}%
\in\mathfrak{p}$. (See the paragraph following equation (\ref{specify}).) Then
since $\left|  z\right|  =\left|  y\right|  $, it follows that $y^{\prime}%
\in\mathcal{G}^{0}$, which yields the equality, $y=z_{y^{\prime}}$. (See
equation (\ref{specify}).) But since, $r\left(  y\right)  =r\left(  z\right)
$, it follows that, $r\left(  z\right)  =r\left(  y\right)  =r\left(
z_{y^{\prime}}\right)  =r\left(  z\right)  \cap y^{\prime}$. Thus, the
inclusion, $r\left(  z\right)  \subseteq y^{\prime}$, holds and hence,
$y=z_{y^{\prime}}=z$. (See the paragraph following equation (\ref{specify}).) A
similar argument shows that $z=y$, in the case when, $z=y\cdot z^{\prime}$ for
some $z^{\prime}\in\mathfrak{p}$.

Next recall that if $\gamma$ is an infinite path, then $\mathcal{A}_{\gamma
}=\{\left(  x,x\right)  \in E\left(  S_{\mathcal{G}}\right)  :\gamma
=x\cdot\gamma^{\prime},\gamma^{\prime}\in\mathfrak{p}^{\infty},s\left(
\gamma^{\prime}\right)  \in r\left(  x\right)  \}$. So evidently if
$\gamma=\gamma^{\prime}$ then $\mathcal{A}_{\gamma}=\mathcal{A}_{\gamma
^{\prime}}$. Suppose, conversely, that $\mathcal{A}_{\gamma}=\mathcal{A}%
_{\gamma^{\prime}}$ and write $\gamma=e_{1}e_{2}\ldots$ and $\gamma^{\prime
}=e_{1}^{\prime}e_{2}^{\prime}\ldots$. For $i\geq1$, and write $\alpha
_{i}:=e_{1}\ldots e_{i}$ and write $\alpha_{i}^{\prime}:=e_{1}^{\prime}\ldots
e_{i}^{\prime}$. Then, of course, $\left|  \alpha_{i}\right|  =\left|
\alpha_{i}^{\prime}\right|  =i$. Since $\mathcal{A}_{\gamma}=\mathcal{A}%
_{\gamma^{\prime}}$, we see that $\left(  \alpha_{i},\alpha_{i}\right)
\in\mathcal{A}_{\gamma^{\prime}}$ and $\left(  \alpha_{i}^{\prime},\alpha
_{i}^{\prime}\right)  \in\mathcal{A}_{\gamma}$. Hence the equations,
$\alpha_{i}^{\prime}e_{i+1}^{\prime}\ldots=\gamma^{\prime}=\alpha_{i}%
\eta^{\prime}$ and $\alpha_{i}e_{i+1}\ldots=\gamma=\alpha_{i}^{\prime}\eta$,
hold for some infinite paths $\eta$ and $\eta^{\prime}$. But since $\left|
\alpha_{i}\right|  =\left|  \alpha_{i}^{\prime}\right|  $ it follows that
$\alpha_{i}=\alpha_{i}^{\prime}$. Since $i\geq1$ was fixed but arbitrary, it
follows that $e_{i}=e_{i}^{\prime}$ for each $i$. Hence $\gamma=\gamma
^{\prime}$. The last assertion is clear, since if $\gamma$ is an infinite
path, then $\mathcal{A}_{\gamma}$ contains elements $(x,x)$ with $|x|$
arbitrarily large.
\end{proof}

The unit space $H_{\mathcal{G}}^{\left(  0\right)  }=\widehat{E\left(
S_{\mathcal{G}}\right)  }$, of the universal groupoid $H=H_{\mathcal{G}}$ of
$S=S_{\mathcal{G}}$ has an explicit parametrization given by the following proposition.

\begin{proposition}
\label{unitspace} The set of semicharacters on $E\left(  S_{\mathcal{G}%
}\right)  $, $H_{\mathcal{G}}^{\left(  0\right)  }=\widehat{E\left(
S_{\mathcal{G}}\right)  }$, may be identified with the disjoint union
$\mathfrak{p}\cup\mathfrak
{p}^{\infty}\cup\left\{  \omega\right\}  $.
\end{proposition}

\begin{proof}
Let $\chi\in\widehat{E\left(  S_{\mathcal{G}}\right)  }$ and recall that
$E\left(  S_{\mathcal{G}}\right)  =\left\{  \left(  x,x\right)  :x\in\mathfrak
{p}\right\}  \cup\left\{  \omega\right\}  $. If $\chi\left(  \omega\right)
=1$, then since $\left(  x,x\right)  \omega=\omega$ for all $\left(
x,x\right)  \in E\left(  S_{\mathcal{G}}\right)  $, we see that $\chi\left(
x,x\right)  =1$ for all $(x,x)$. That is, $\chi$ is the constant non-zero
semicharacter on $E\left(  S_{\mathcal{G}}\right)  $. So the filter in
$E\left(  S_{\mathcal{G}}\right)  $ determined by $\chi$, $\mathcal{A}_{\chi}%
$, is simply $\tilde{\omega}$. Thus $\chi=\omega$. So we may suppose that
$\chi\neq\omega$. Then $\chi\left(  \omega\right)  =0$. Let $M:=\{|x|:(x,x)\in
\mathcal{A}_{\chi}\}$. The strategy here is the following. We will parametrize
the filter in $E\left(  S_{\mathcal{G}}\right)  $, $\mathcal{A}_{\chi
}=\left\{  \left(  x,x\right)  \in E\left(  S_{\mathcal{G}}\right)
:\chi\left(  x,x\right)  =1\right\}  $, (\cite[p.173-174]{APB}) by showing
that $\mathcal{A}_{\chi}=\mathcal{A}_{y}$, for an ultrapath $y\in\mathfrak{p}$
if $M$ is finite and by showing that $\mathcal{A}_{\chi}=\mathcal{A}_{\gamma}$
for a suitable $\gamma\in\mathfrak{p}^{\infty}$ if $M$ is infinite. Then we
identify $\chi$ either with $y$ or $\gamma$. (See Remark \ref{equality of
filters}.) To begin recall that the set, $M$, is defined by the equation,
$M=\{|x|:(x,x)\in\mathcal{A}_{\chi}\}$.  {}

\begin{enumerate}
\item [Case I]$M$ is finite. In this case there an ultrapath $y$ in
$\mathfrak{p}$ so that $\left(  y,y\right)  \in\mathcal{A}_{\chi}$ and so that
$\left|  y\right|  =\max M$. We show that $\mathcal{A}_{y}=\mathcal{A}_{\chi}%
$. For this end, take any $\left(  x,x\right)  \in$ $\mathcal{A}_{y}$. Then by
the definition of $\mathcal{A}_{y}$, the inequality, $\left(  x,x\right)
\left(  y,y\right)  \neq\omega$, holds, and either $|y|>|x|$, or if $|y|=|x|$,
then the inclusion, $r\left(  y\right)  \subseteq r\left(  x\right)  $, holds.
(See the paragraph following equations (\ref{ultras}) and (\ref{infinitep}).)
Hence by Remark \ref{order in E}, the inequality, $\left(  x,x\right)
\geq\left(  y,y\right)  $, holds. This means that, $\left(  x,x\right)
\left(  y,y\right)  =\left(  y,y\right)  $, which implies the equation
$\chi\left(  x,x\right)  \chi\left(  y,y\right)  =\chi\left(  y,y\right)  $.
(See previous paragraph to Remark \ref{order in E}) Since $\left(  y,y\right)
$ belongs to $\mathcal{A}_{\chi}$, we see that $\chi\left(  x,x\right)  =1$.
Hence, $\left(  x,x\right)  \in\mathcal{A}_{\chi}$, and $\mathcal{A}%
_{y}\subseteq\mathcal{A}_{\chi}$. On the other hand suppose that $\left(
x,x\right)  \in$ $\mathcal{A}_{\chi}$. Since $\chi\left(  \omega\right)  =0$,
it follows that the product, in $\mathcal{A}_{\chi}$, $\left(  x,x\right)
\left(  y,y\right)  $, is not $\omega$. (Recall that $\left(  y,y\right)
\in\mathcal{A}_{\chi}$) Moreover, the inequality, $\left|  y\right|
\geq\left|  x\right|  $, holds, since $\left|  y\right|  =\max M$. This yields
the equation, $y=x\cdot y^{\prime}$, for some $y^{\prime}\in\mathfrak{p}$. If
$\left|  y\right|  =\left|  x\right|  $, then $y^{\prime}\in\mathcal{G}^{0}$
and hence we must have $y=x_{y^{\prime}}$. But then $r\left(  y\right)
=r\left(  x_{y^{\prime}}\right)  =r\left(  x\right)  \cap y^{\prime}\subseteq
r\left(  x\right)  $. (See equation (\ref{specify}) in Notation
\ref{UltrapathAlgebra}.) Thus $\left(  x,x\right)  \in\mathcal{A}_{y}$ and
hence $\mathcal{A}_{\chi}\subseteq\mathcal{A}_{y}$. Thus $\mathcal{A}_{\chi
}=\mathcal{A}_{y}$. (Note that Remark \ref{equality of filters} shows that $y$
is uniquely determined by $\chi$.

\item[Case II] $M$ is infinite. We'll show that there is a path $\gamma
\in\mathfrak{p}^{\infty}$ such that $\mathcal{A}_{\chi}=\mathcal{A}_{\gamma}$.
Indeed, take a pair $\left(  x_{1},x_{1}\right)  \in\mathcal{A}_{\chi}$ such
that $\left|  x_{1}\right|  >0$. Since the set $M$ is countably infinite, we
may find another pair $\left(  x_{2},x_{2}\right)  \in\mathcal{A}_{\chi}$,
such that $\left|  x_{2}\right|  >\left|  x_{1}\right|  $. Moreover, since the
product $\left(  x_{1},x_{1}\right)  \left(  x_{2},x_{2}\right)  $ is not
$\omega$, we may write $x_{2}=x_{1}\cdot y_{2}$, for some $y_{2}\in
\mathfrak{p}$ such that $\left|  y_{2}\right|  >0$. Using the same argument
for the pair $\left(  x_{2},x_{2}\right)  $, we may find another pair,
$\left(  x_{3},x_{3}\right)  \in\mathcal{A}_{\chi}$, such that $x_{3}%
=x_{2}\cdot y_{3}$ for some $y_{3}\in\mathfrak{p}$ with $\left|  y_{3}\right|
>0$. Continuing this process inductively, we obtain a sequence of pairs
$\left\{  \left(  x_{i},x_{i}\right)  \right\}  _{i\geq1}$ in $\mathcal{A}%
_{\chi}$ so that each $x_{i}$ has positive length and $x_{i+1}=x_{i}\cdot
y_{i+1}$, where $y_{i}\in\mathfrak{p}$ with $\left|  y_{i}\right|  >0$, for
all $i$. So, if we set $x_{i}=\left(  \alpha_{i},A_{i}\right)  $ and
$y_{i}=\left(  \beta_{i},B_{i}\right)  $, we may use the relation,
$x_{i+1}=x_{i}\cdot y_{i+1}$, to define an infinite path $\gamma$ in
$\mathfrak{p}^{\infty}$ by the equation $\gamma:=\alpha_{1}\beta_{2}\beta
_{3}\ldots$. We show that $\mathcal{A}_{\chi}=\mathcal{A}_{\gamma}$. For this
end, take any $\left(  x,x\right)  $ in $\mathcal{A}_{\chi}$. Then the product
in $\mathcal{A}_{\chi}$, $\left(  x,x\right)  \left(  x_{i},x_{i}\right)  $,
is not $\omega$ for any $i$. Since the set, $\{\left|  x_{i}\right|
:i=1,\ldots\}$ is unbounded above, there is a positive integer $i_{0}$, such
that $\left|  x_{i_{0}}\right|  >\left|  x\right|  $. Then by definition, we
have $x_{i_{0}}=x\cdot x_{i_{0}}^{\prime}$ for some $x_{i_{0}}^{\prime}%
\in\mathfrak{p}$ with positive length. But by setting $x_{i_{0}}^{\prime
}=(\alpha_{i_{0}}^{\prime},A_{i_{0}}^{\prime})$, we have $\gamma=x_{i_{0}%
}\cdot\beta_{i_{0}+1}\ldots=x\cdot\alpha_{i_{0}}^{\prime}\beta_{i_{0}+1}%
\ldots$. (Note that $\gamma=x_{i}\cdot\beta_{i+1}\beta_{i+2}\ldots$, for each
$i\geq1$ by the paragraph before Definition \ref{inverse semigroup}.) Thus
$\left(  x,x\right)  \in$ $\mathcal{A}_{\gamma}$ and hence $\mathcal{A}_{\chi
}\subseteq\mathcal{A}_{\gamma}$. For the reverse inclusion, let $\left(
x,x\right)  \in$ $\mathcal{A}_{\gamma}$. Then $\gamma=x\cdot\gamma^{\prime}$
for some infinite path $\gamma^{\prime}$ in $\mathfrak{p}^{\infty}$. So in
this situation we always may choose a path $x_{i_{0}}$ from the sequence
$\left\{  \left(  x_{i},x_{i}\right)  \right\}  _{i\geq1}$, so that the
product in $E\left(  S_{\mathcal{G}}\right)  $, $\left(  x,x\right)  \left(
x_{i_{0}},x_{i_{0}}\right)  $, is not $\omega$ and $\left|  x_{i_{0}}\right|
>\left|  x\right|  $. Then since $\gamma:=\alpha_{1}\beta_{2}\beta_{3}%
\ldots=x_{i}\cdot\beta_{i+1}\ldots$ for each $i\geq1$, Remark \ref{order in E}
shows that the inequality, $\left(  x,x\right)  >\left(  x_{i_{0}},x_{i_{0}%
}\right)  $. This, in turn, yields the equation, $\left(  x,x\right)  \left(
x_{i_{0}},x_{i_{0}}\right)  =\left(  x_{i_{0}},x_{i_{0}}\right)  $. It follows
that $\chi\left(  x,x\right)  \chi\left(  x_{i_{0}},x_{i_{0}}\right)
=\chi\left(  x_{i_{0}},x_{i_{0}}\right)  $, and since $\chi\left(  x_{i_{0}%
},x_{i_{0}}\right)  =1$, it follows that $\chi\left(  x,x\right)  =1$. Thus
$\left(  x,x\right)  \in\mathcal{A}_{\chi}$, showing that $\mathcal{A}%
_{\gamma}\subseteq\mathcal{A}_{\chi}$. Hence $\mathcal{A}_{\chi}%
=\mathcal{A}_{\gamma}$. Again, we may appeal to Remark \ref{equality of
filters} to guarantee that the infinite path $\gamma$ is uniquely determined
by $\chi$.
\end{enumerate}
\end{proof}

Recall that the topology on $\widehat{E\left(  S_{\mathcal{G}}\right)  }$ is
the topology of pointwise convergence on $E\left(  S_{\mathcal{G}}\right)  $,
and so the family, $\{D_{\left(  x,x\right)  }:\left(
x,x\right)  \in E\left(  S_{\mathcal{G}}\right)  \}$, of compact open sets forms a subbasis for the
topology. In this setting, the subbasic set, $D_{\left(  x,x\right)  }$, is
given by the equation, $D_{\left(  x,x\right)  }=\{y\in\mathfrak{p}:\left(
x,x\right)  \left(  y,y\right)  \neq\omega$; $|y|>|x|$ or, if $|y|=|x|$, then
$r\left(  y\right)  \subseteq r\left(  x\right)  \}\cup\{\gamma\in\mathfrak
{p}^{\infty}:\gamma=x\cdot\gamma^{\prime}$, $\gamma^{\prime}\in\mathfrak
{p}^{\infty},s\left(  \gamma^{\prime}\right)  \in r\left(  x\right)  \}$. (See
\cite[Chap.4, p.174]{APB} and equations (\ref{ultras}) and (\ref{infinitep}).)
We would like a more concrete description of the topology on $\widehat
{E\left(  S_{\mathcal{G}}\right)  }$. For this purpose, it is convenient to
introduce the following notation and definition.

\begin{notation}
Let $F$ be a finite subset of $E\left(  S_{\mathcal{G}}\right)  $. Let
$D_{\left(  x,x\right)  ;F}:=D_{\left(  x,x\right)  }\cap\bigcap
\limits_{\left(  z,z\right)  \in F}D_{\left(  z,z\right)  }^{c}$. Since we are
only interested in non-empty basic sets, we may suppose that $\left(
z,z\right)  <\left(  x,x\right)  $, for all $\left(  z,z\right)  \in F$. See
\cite[p.174]{APB}.
\end{notation}

\begin{definition}
Let $\mathcal{G}=\left(  G^{0},\mathcal{G}^{1},r,s\right)  $ be an ultragraph
and let $A$ be a subset of $G^{0}$. We say that the edge $e\in\mathcal{G}^{1}$
is emitted by $A$ whenever $s\left(  e\right)  \in A$.
\end{definition}

\begin{lemma}
\label{topology} Given an ultrapath $y\in D_{\left(  x,x\right)  ,F}$ there is
a finite set $K$ of edges emitted by the range of $y$, $r\left(  y\right)  $,
and a finite subcollection $Q$ of $\mathcal{G}^{0}$, such that no set in $Q$
contains  $r\left(  y\right)  $ and such that $y\in
D_{\left(  y,y\right)  ;K,Q}\subset D_{\left(  x,x\right)  ,F}$ where,
\[
D_{\left(  y,y\right)  ;K,Q}:=D_{\left(  y,y\right)  }\cap\bigcap
\limits_{e{\in}K}D_{\left(  y\cdot e,y\cdot e\right)  }^{c}\cap
\bigcap\limits_{C{\in}Q}D_{\left(  y_{C},y_{C}\right)  }^{c}\text{.}%
\]
\end{lemma}

\begin{proof}
Recall that, the open subset of $\widehat{E\left(  S_{\mathcal{G}}\right)  }$,
$D_{\left(  x,x\right)  }$, is given by the equation, $D_{\left(  x,x\right)
}=\{ y\in\mathfrak{p}:\left(  x,x\right)  \left(  y,y\right)  \neq\omega$;
$|y|>|x|$ or, if $|y|=|x|$, then $r\left(  y\right)  \subseteq r\left(
x\right)  \}\cup\{\gamma\in\mathfrak{p}^{\infty}:\gamma=x\cdot\gamma^{\prime}%
$, $\gamma^{\prime}\in\mathfrak{p}^{\infty},s\left(  \gamma^{\prime}\right)
\in r\left(  x\right)  \}$. We leave to the reader to check that the
ultrapath, $y$, lies in $D_{\left(  y,y\right)  ;K,Q}$. Suppose $y\in
D_{\left(  x,x\right)  ,F}$ and fix any $(z,z)\in F$, such that $\left(
z,z\right)  <\left(  x,x\right)  $. Then since $y\in D_{\left(  x,x\right)
,F}$, we have, $\left(  x,x\right)  \left(  y,y\right)  \neq\omega$, and so
either $|y|>|x|$ or, if $|y|=|x|$, then $r\left(  y\right)  \subseteq r\left(
x\right)  $ and $y\notin D_{\left(  z,z\right)  }$. Also since $\left(
z,z\right)  <\left(  x,x\right)  $, we have $|z|>|x|$ or, if $|z|=|x|$, then
$r\left(  z\right)  \subsetneq r\left(  x\right)  $. (See Remark \ref{order in
E}.) We have the following cases:

\begin{enumerate}
\item [Case I]$|z|\geq1$. Since $y\notin D_{\left(  z,z\right)  }$, if the
equation, $\left(  y,y\right)  \left(  z,z\right)  =\omega$, holds, we have
$D_{\left(  y,y\right)  }\subseteq D_{\left(  x,x\right)  }\cap D_{\left(
z,z\right)  }^{c}$; otherwise, either the inequality, $|y|<|z|\,$, holds or, if
$|y|=|z|$, then the range of $z$, $r\left(  z\right)  $, does not contain the
range of $y$, $r\left(  y\right)  $. In this situation, we have $z=y\cdot
z^{\prime}$, for some $z^{\prime}\in\mathfrak{p}$. So we have to consider two
subcases in this first case.

\item[I(1)] $|z^{\prime}|\geq1$. Let $e^{\prime}$ be the initial edge in
$\mathcal{G}^{1}$ of $z^{\prime}$. So $s\left(  e^{\prime}\right)  \in
r\left(  y\right)  $, and we see that $D_{\left(  y,y\right)  }\cap D_{(y\cdot
e^{\prime},y\cdot e^{\prime})}^{c}\subseteq D_{\left(  x,x\right)  }\cap
D_{\left(  z,z\right)  }^{c}$.

\item[I(2)] $z^{\prime}\in\mathcal{G}^{0}$. In this case, the equality,
$|z|=|y|$, holds. So since $y\notin D_{\left(  z,z\right)  }$, the set
$r\left(  z\right)  =r\left(  y\right)  \cap z^{\prime}$, does not contain the
range of $y$, $r\left(  y\right)  $. This implies that $z^{\prime}$ can not
contain the range of $y$, $r\left(  y\right)  $. Then we see that $D_{\left(
y,y\right)  }\cap D_{(y_{z^{\prime}},y_{z^{\prime}})}^{c}\subseteq D_{\left(
x,x\right)  }\cap D_{\left(  z,z\right)  }^{c}$.

\item[Case II] $z\in\mathcal{G}^{0}$. In this case, $x\in\mathcal{G}^{0}$ and
$z\subsetneq x$. But since $y\notin D_{\left(  z,z\right)  }$, $s\left(
y\right)  \nsubseteq z$. Then we have, $D_{\left(  y,y\right)  }\subseteq
D_{\left(  x,x\right)  }\cap D_{\left(  z,z\right)  }^{c}$.
\end{enumerate}

In any of the above cases, we may take the set $K$ to be the union when
$\left(  z,z\right)  $ runs over the set $F$ of the sets $\{ e_{z}^{\prime}%
\in\mathcal{G}^{1}$: $e_{z}^{\prime}$ is the initial edge of $z^{\prime}$,
$z=y\cdot z^{\prime}\}$; while for the set $Q$, we may take the set, $\{
z^{\prime}\in\mathcal{G}^{0}$: $z=y_{z^{\prime}}$, $\left(  z,z\right)  \in
F\}$. Thus $y\in D_{\left(  y,y\right)  ;K,Q}\subset D_{\left(  x,x\right)
,F}$.
\end{proof}

The following lemma describes the topology on the set of infinite paths
$\mathfrak{p}^{\infty}$.

\begin{lemma}
\label{topoinf}A neighborhood basis for $\gamma\in\mathfrak{p}^{\infty}$ is
given by the sets of the form $D_{\left(  y,y\right)  }$, where $y=\left(
\beta,B\right)  $ and $\beta$ is an initial segment \ of $\gamma$.
\end{lemma}

\begin{proof}
Let $\gamma\in\mathfrak{p}^{\infty}$ and $\gamma\in D_{\left(  x,x\right)
;F}:=D_{\left(  x,x\right)  }\cap\bigcap\limits_{\left(  z,z\right)  \in
F}D_{\left(  z,z\right)  }^{c}$. Since $\gamma\in D_{(x,x)}$, $\gamma
=x\cdot\gamma^{\prime}$, where $\gamma^{\prime}\in\mathfrak{p}^{\infty}$ and
$s\left(  \gamma^{\prime}\right)  \in r\left(  x\right)  $. Also since we are
interested in nonempty basis elements, we may assume that the inequality,
$\left(  x,x\right)  >\left(  z,z\right)  $, holds. So we have $z=x\cdot
z^{\prime}$, $z^{\prime}\in\mathfrak{p}$. Since $\gamma^{\prime}$ is an
infinite path and $\gamma\notin D_{\left(  z,z\right)  }$, we may choose an
initial segment with positive length, $y^{\prime}$, of $\gamma^{\prime}$ so
that $\left\vert y^{\prime}\right\vert >\left\vert z^{\prime}\right\vert $.
Set $y:=x\cdot y^{\prime}$ and notice then $\gamma\in D_{\left(  y,y\right)
}\subset D_{\left(  x,x\right)  }\cap D_{\left(  z,z\right)  }^{c}$.
\end{proof}

The universal groupoid $H_{\mathcal{G}}$ for $S_{\mathcal{G}}$ has an explicit
parametrization given by the following theorem.

\begin{theorem}
\label{univ.groupoid} The universal groupoid $H_{\mathcal{G}}$ for
$S_{\mathcal{G}}$ can be identified with the union of $\left\{  \omega
\right\}  $ and the set of all triples of the form $\left(  x\cdot\mu,\left|
x\right|  -\left|  y\right|  ,y\cdot\mu\right)  $ where $x$, $y\in\mathfrak
{p}$, $r\left(  x\right)  =r\left(  y\right)  $, $\mu\in\mathfrak{p}%
\cup\mathfrak
{p}^{\infty}$, and $x\cdot\mu$, $y\cdot\mu\in\mathfrak{p}\cup\mathfrak
{p}^{\infty}$. Multiplication on $H_{\mathcal{G}}$ is given by the formula:%
\[
\left(  x\cdot\mu,\left|  x\right|  -\left|  y\right|  ,y\cdot\mu\right)
\left(  y\cdot\mu,\left|  y\right|  -\left|  y^{\prime}\right|  ,y^{\prime
}\cdot\mu^{\prime}\right)  :=\left(  x\cdot\mu,\left|  x\right|  -\left|
y^{\prime}\right|  ,y^{\prime}\cdot\mu^{\prime}\right)  \text{,}%
\]
and inversion is given by the formula,
\[
\left(  x\cdot\mu,\left|  x\right|  -\left|  y\right|  ,y\cdot\mu\right)
^{-1}:=\left(  y\cdot\mu,\left|  y\right|  -\left|  x\right|  ,x\cdot
\mu\right)  \text{.}%
\]
The canonical map $\Lambda:S_{\mathcal{G}}\longrightarrow H^{a}$, sends
$\omega$ to $\{\omega\}$ and any element $\left(  x,y\right)  \in
S_{\mathcal{G}}$ to the (compact open) set $\mathcal{A}\left(  x,y\right)  $,
where
\[
\mathcal{A}\left(  x,y\right)  :=\left\{  \left(  x\cdot\mu,|x|-|y|,y\cdot
\mu\right)  :\mu\in\mathfrak{p}\cup\mathfrak{p}^{\infty}\right\}  \text{.}%
\]
Further $\mathcal{A}\left(  x,x\right)  =D_{\left(  x,x\right)  }$, for each
$x\in\mathfrak{p}$. The locally compact groupoid $H_{\mathcal{G}}$ is Hausdorff.
\end{theorem}

\begin{proof}
The proof is close to that for the Cuntz semigroup $S_{n}$ in \cite[p.182-186]%
{APB}. Also see \cite[p.9-10]{APP}. We have to compute the equivalence classes
$\overline{\left(  \chi,s\right)  }$, $\chi\in D_{s}$. We suppose first
$\chi=u$, is an ultrapath and then the computation of the equivalence class
$\overline{\left(  \chi,s\right)  }$, when $\chi$ is an infinite path is
similar. So let $\chi=u$, and let $s=\left(  x,y\right)  $ and $t=\left(
z,w\right)  $ be elements in $S_{\mathcal{G}}$ such that $\chi\in D_{\left(
x,x\right)  \left(  z,z\right)  }$. Then $u=x\cdot u^{\prime}$ and $u=z\cdot
u^{\prime\prime}$. If we let $u=(u_{1}\cdots u_{|u|},U_{\left|  u\right|  })$,
then without lost of generality, we may assume that $x=u_{1}\cdots u_{m}$, and
$z=u_{1}\cdots u_{r}$, where $m\leq r\leq|u|$. Thus $s=\left(  u_{1}\cdots
u_{m},y\right)  $ and $t=\left(  u_{1}\cdots u_{r},w\right)  $. Let $e\in
E\left(  S_{\mathcal{G}}\right)  $ be such that $\chi=u\in D_{e}$,
$e\leq\left(  ss^{\ast}\right)  \left(  tt^{\ast}\right)  =\left(  x,x\right)
\left(  z,z\right)  =\left(  z,z\right)  $ and $es=et$. Then $e=\left(
f,f\right)  $, where $f$ is such that $u=f\cdot u^{\prime\prime\prime}$,
$u^{\prime\prime\prime}\in\mathfrak{p}$ and $f=z\cdot f^{\prime\prime}$,
$f^{\prime\prime}\in\mathfrak{p}$. So we can write $e=\left(  f,f\right)  $
$=(u_{1}\cdots u_{r^{\prime}},u_{1}\cdots u_{r^{\prime}})$, where $m\leq r\leq
r^{\prime}\leq|u|$. Then we have
\begin{align*}
&  (\left(  u_{1}\cdots u_{r^{\prime}},y\cdot u_{m+1}\cdots u_{r^{\prime}%
}\right) \\
&  =\left(  f,y\cdot f^{\prime}\right)  =es=et\\
&  =\left(  f,w\cdot f^{\prime\prime}\right)  =(\left(  u_{1}\cdots
u_{r^{\prime}},w\cdot u_{r+1}\cdots u_{r^{\prime}}\right)  \text{.}%
\end{align*}
This means that $w=y\cdot u_{1}\cdots u_{r^{\prime}}$. To link our groupoid
with Renault's model for the Cuntz groupoid $G_{n}$, described in
\cite[Section 4.2, Example 3]{APB}, we associate the pair
\[
\left(  \chi,s\right)  =((u_{1}\cdots u_{m}u_{m+1}\cdots u_{|u|},U_{\left|
u\right|  }),(u_{1}\cdots u_{m},y))\text{,}%
\]
with the triple
\[
((u_{1}\cdots u_{m}u_{m+1}\cdots u_{|u|},U_{\left|  u\right|  }),m-|y|,y\cdot
(u_{m+1}\cdots u_{|u|},U_{\left|  u\right|  }))\text{.}%
\]
The argument is reversible and shows that this map is a bijection.

We now prove that $H_{\mathcal{G}}$ is Hausdorff, leaving the remaining
verifications of the theorem to the reader. Let $a=\left(  x,|x|-|y|,y\right)
$, $b=\left(  x^{\prime},|x^{\prime}|-|y^{\prime}|,y^{\prime}\right)  $ belong
to $H_{\mathcal{G}}$ with $a\neq b$. If $\mathcal{A}\left(  x,y\right)
\cap\mathcal{A}\left(  x^{\prime},y^{\prime}\right)  =\emptyset$, then we can
separate $a$ and $b$ using $\mathcal{A}\left(  x,y\right)  $ and
$\mathcal{A}\left(  x^{\prime},y^{\prime}\right)  $. Suppose that
$\mathcal{A}\left(  x,y\right)  \cap\mathcal{A}\left(  x^{\prime},y^{\prime
}\right)  \neq\emptyset$. Then there exist $\chi$, $\chi^{\prime}$ such that
\[
\left(  x\cdot\chi,|x|-|y|,y\cdot\chi\right)  =\left(  x^{\prime}\cdot
\chi^{\prime},|x^{\prime}|-|y^{\prime}|,y^{\prime}\cdot\chi^{\prime}\right)
\text{.}%
\]
Then the equation, $|x|-|y|=|x^{\prime}|-|y^{\prime}|$, holds. Furthermore the
ultrapaths $x$, $x^{\prime}$ and $y$, $y^{\prime}$ are comparable. We can
suppose that for some $u\in\mathfrak{p}$, $x^{\prime}=x\cdot u$. Then
$y^{\prime}=y\cdot w$, $w\in\mathfrak{p}$, where $|u|=|w|$, and since $u$ and
$w$ are initial segments of $\chi$, $u=w$. Then $x^{\prime}=x\cdot u$ and
$y^{\prime}=y\cdot u$. If $|u|>0$ then $\mathcal{A}\left(  x,y\right)
\cap\mathcal{A}^{c}\left(  x^{\prime},y^{\prime}\right)  $ and $\mathcal{A}%
\left(  x^{\prime},y^{\prime}\right)  $ separate $a$ and $b$. If $|u|=0$ then,
$|x^{\prime}|=|x|$ and $|y^{\prime}|=|y|$. But since $a\neq b$, we have
$r\left(  x^{\prime}\right)  \subsetneq r\left(  x\right)  $ and $r\left(
y^{\prime}\right)  \subsetneq r\left(  y\right)  $. So we see that
$\mathcal{A}\left(  x,y\right)  \cap\mathcal{A}^{c}\left(  x^{\prime
},y^{\prime}\right)  $ and $\mathcal{A}\left(  x^{\prime},y^{\prime}\right)  $
separate $a$ and $b$ as well. So $H_{\mathcal{G}}$ is Hausdorff.
\end{proof}

\begin{remark}
\label{isolated}The singleton $\{\omega\}$ is a clopen invariant subset of
$H_{\mathcal{G}}^{(0)}$. The reduction $H_{\mathcal{G}}^{(0)}\vert
_{\{\omega\}}$ is simply $\{\omega\}$. Consequently, the interesting part of
$H_{\mathcal{G}}$ is $H_{\mathcal{G}}\vert_{\mathfrak{p}\cup\mathfrak
{p}^{\infty}}$.
\end{remark}

At this point, we have just identified the universal groupoid $H_{\mathcal{G}%
}$ for $S_{\mathcal{G}}$. However, we still need to find the correct groupoid
for $\mathcal{G}$. That is, we want to find a groupoid $\mathfrak
{G}_{\mathcal{G}}$ such that $C^{\ast}\left(  \mathcal{G}\right)  \backsimeq
C^{\ast}\left(  \mathfrak{G}_{\mathcal{G}}\right)  $. (See Definition
\ref{cuntz-pimsner rep.} and Remark \ref{principal difficulty}.) To do this we
shall take a closer look at the unit space of $H_{\mathcal{G}}$, which is the
disjoint union $\mathfrak{p}\cup\mathfrak{p}^{\infty}\cup\left\{
\omega\right\}  $, and use the concept of \emph{ultrafilters} to investigate
it.\textit{\ }(See \cite{Samuel}.) Indeed, we give $G^{0}$ the discrete
topology. Then the points of the Stone-\v{C}ech compactification $\beta G^{0}$
of $G^{0}$ can be regarded as ultrafilters on $G^{0}$. (See the introduction
in \cite{DS}.) So, consider the subcollection of $\mathcal{G}^{0}$,
$\mathcal{U}\left(  \mathcal{G}^{0}\right)  $, defined to be the collection of
all sets in $\mathcal{G}^{0}$ whose principal filter in $\mathcal{G}^{0}$ is
also an ultrafilter over $G^{0}$. That is, $\mathcal{U}\left(  \mathcal{G}%
^{0}\right)  =\{ A\in\mathcal{G}^{0}:\widetilde{A}\in\beta G^{0}\}$, where
$\widetilde{A}$ is the principal filter in $\mathcal{G}^{0}$ determined by
$A\in\mathcal{G}^{0}$, i.e., $\widetilde{A}=\{ B\in\mathcal{G}^{0}\mid
A\subseteq B\}$. (See the paragraph following Remark \ref{order in E}.)
Observe that $\mathcal{U}\left(  \mathcal{G}^{0}\right)  $ contains every
singleton set determined by the vertices in $G^{0}$. Furthermore, one can
check that a topology for $\mathcal{U}\left(  \mathcal{G}^{0}\right)  $ may
defined by taking the family of subsets of $\mathcal{U}\left(  \mathcal{G}%
^{0}\right)  $, $\{\widehat{A}:A\in\mathcal{G}^{0}\}$, where $\widehat{A}:=\{
B\in\mathcal{U}\left(  \mathcal{G}^{0}\right)  :A\in\widetilde{B}$$\}$, as a
subbasis of closed subsets. (See \cite[first paragraph on page 117 ]{Samuel}.)
An important property of the collection, $\mathcal{U}\left(  \mathcal{G}%
^{0}\right)  $, however, is that, for each member $C$ in $\mathcal{U}\left(
\mathcal{G}^{0}\right)  $, its associated semicharacter, $\chi_{C}$, satisfies
the equations:%
\begin{equation}
\chi_{C}\left(  A\cup B\right)  =\chi_{C}\left(  A\right)  +\chi_{C}\left(
B\right)  -\chi_{C}\left(  A\cap B\right)  \text{ and }\chi_{C}\left(
\emptyset\right)  =0\text{,} \label{additive}%
\end{equation}
for all $A,B\in\mathcal{G}^{0}$. (See the paragraph following Remark
\ref{order in E} and see \cite[p.104]{Samuel}.) Thus, $\mathcal{U}\left(
\mathcal{G}^{0}\right)  $, may be viewed as a closed subset of $\widetilde
{\mathcal{G}^{0}}$ via the map $A\longrightarrow\chi_{A}$. As we shall see,
this observation is critical for building the groupoid model for $\mathcal{G}%
$. (See second part of \textit{(i)} in Definition \ref{cuntz-pimsner rep.} and
Remark \ref{principal difficulty}). We shall call the elements in
$\mathcal{U}\left(  \mathcal{G}^{0}\right)  $ \emph{ultrasets}. Another
crucial tool is the next proposition, which identifies the closure of the set
of all infinite paths $\mathfrak{p}^{\infty}$, $\overline{\mathfrak{p}%
^{\infty}}$. This is the key to obtaining the unit space of our groupoid for
$\mathcal{G}$. (Compare with the first statement of Proposition 4 in
\cite{APP}.) For this purpose, we need the following generalization of the
notion of {}``infinite emitter''\ from the setting of ordinary graphs.

\begin{definition}
\label{infinte emitter}Let $\mathcal{G}=\left(  G^{0},\mathcal{G}%
^{1},r,s\right)  $ be an ultragraph and for each subset $A$ of $G^{0}$, let
$\varepsilon\left(  A\right)  $ be the set $\{ e\in\mathcal{G}^{1}:s\left(
e\right)  \in A\}$. We shall say that a set $A$ in $\mathcal{G}^{0}$ is an
\emph{infinite emitter} whenever $\varepsilon\left(  A\right)  $ is infinite.
\end{definition}

\begin{proposition}
\label{closure} The set of infinite paths, $\mathfrak{p}^{\infty}$, is dense in
$Y_{\infty}\cup\mathfrak{p}^{\infty}$, where $Y_{\infty}$ is defined to be the
set of all ultrapaths $y$ in $\mathfrak{p}$ whose range $r\left(  y\right)  $
is an ultraset emitting infinitely many edges.
\end{proposition}

\begin{proof}
Take any $\chi$ in the closure $\overline{\mathfrak{p}^{\infty}}$ in
$\widehat{E\left(  S_{\mathcal{G}}\right)  }=\mathfrak{p}\cup\mathfrak
{p}^{\infty}\cup\left\{  \omega\right\}  $. Then there is an infinite sequence
$\left\{  \gamma_{i}\right\}  _{i\geq1}$ in $\mathfrak{p}^{\infty}$ such that
$\gamma_{i}\longrightarrow\chi$. If $\chi$ is not an infinite path, then by
Remark \ref{isolated} it must be an ultrapath, say $\chi=y$. So for large $i$,
each $\gamma_{i}=y\cdot\gamma_{i}^{\prime}$ where, $\{ s(\gamma_{i}^{\prime
})\}\longrightarrow r\left(  y\right)  $ eventually. (See Lemma \ref{topoinf}%
.) Since $\mathcal{U}\left(  \mathcal{G}^{0}\right)  $ is closed and since
each $\{ s(\gamma_{i}^{\prime})\}$ lies in $\mathcal{U}\left(  \mathcal{G}%
^{0}\right)  $, it follows that $r\left(  y\right)  $ belongs to
$\mathcal{U}\left(  \mathcal{G}^{0}\right)  $. Therefore $r\left(  y\right)
\in\widehat{r\left(  y\right)  }$. Moreover, since $\{ s(\gamma_{i}^{\prime
})\}\longrightarrow r\left(  y\right)  $ eventually, it follows that $\left\{
s(\gamma_{i}^{\prime})\right\}  \in\widehat{r\left(  y\right)  }$, for
infinitely many $i$'s. Thus $r\left(  y\right)  $ is an infinite emitter. For
the reverse inclusion, take any $y$ in $Y_{\infty}$. Then by the definition of
$Y_{\infty}$, the range of $y$, $r\left(  y\right)  $, is an ultraset emitting
infinitely many edges. Suppose that the ultrapath $y$ belongs to the open set
in $\widehat{E\left(  S_{\mathcal{G}}\right)  }$, $D_{\left(  y,y\right)
;K,Q}$, which is,
\[
D_{\left(  y,y\right)  }\cap\bigcap\limits_{e{\in}K}D_{\left(  y\cdot
e,y\cdot e\right)  }^{c}\cap\bigcap\limits_{C{\in}Q}D_{\left(
y_{C},y_{C}\right)  }^{c}\text{,}%
\]
where the set $K$ is a finite set of edges emitted by $r\left(  y\right)  $,
and $Q$ is a finite subcollection of $\mathcal{G}^{0}$ consisting of sets that
do not contain the range of $y$, $r\left(  y\right)  $. (See \cite[Chap.4,
p.174]{APB} and Lemma \ref{topology}.) Then no set $C$ in $Q$ belongs to
$\widetilde{r(y)}$. (Recall that, $\widetilde{r(y)}=\left\{  A\in
\mathcal{G}^{0}:r\left(  y\right)  \subseteq A\right\}  $.) Fix a set $C$ in
$Q$. Then since $C\notin\widetilde{r(y)}$ and since $\widetilde{r(y)}$ is an
ultrafilter on $G^{0}$, the complement of $C$, $C^{c}$, must belong to
$\widetilde{r(y)}$. (See \cite[Theorem IV, p.107]{Samuel}.) That is, $r(y)$
and $C$ are disjoint. But by hypothesis the set, $r\left(  y\right)  $, is an
infinite emitter. Consequently, the set, $\varepsilon\left(  r\left(
y\right)  \right)  $, is infinite. So we always may choose an edge $e_{1}$ in
$\mathcal{G}^{1}$, such that $e_{1}\notin K$ and $s\left(  e_{1}\right)  \in
r\left(  y\right)  \subseteq C^{c}$. Since we are assuming that the ultragraph
$\mathcal{G}$ has no sinks, we may choose another edge, say $e_{2}$, so that
$s\left(  e_{2}\right)  \in r\left(  e_{1}\right)  $. Inductively, we may form
an infinite path, $\gamma:=e_{1}e_{2}\ldots$, so that $s\left(  \gamma\right)
\in r\left(  y\right)  \subseteq C^{c}$. Since $C$ was fixed but arbitrary, we
may conclude that the source of $\gamma$, $s\left(  \gamma\right)  $, belongs
to $r\left(  y\right)  $ but not to any set $C$ in $Q$. So setting
$\gamma^{\prime}=y\cdot\gamma$, we see that $\gamma^{\prime}\in D_{\left(
y,y\right)  ;K,Q}$, and hence, we may conclude that $y$ lies in the closure of
$\mathfrak
{p}^{\infty}$, $\overline{\mathfrak{p}^{\infty}}$.
\end{proof}

We next set $X:=Y_{\infty}\cup\mathfrak{p}^{\infty}\subset H_{\mathcal{G}%
}^{\left(  0\right)  }$. By Proposition \ref{closure}, $X$ is a closed subset of the unit
space $H_{\mathcal{G}}^{\left(  0\right)  }=\mathfrak{p}\cup\mathfrak
{p}^{\infty}\cup\left\{  \omega\right\}  $. Hence $X$ is a locally compact
Hausdorff space. Furthermore, since for every
triple, $\left(  x\cdot\mu,\left|  x\right|  -\left|  y\right|  ,y\cdot
\mu\right)  $, in $H_{\mathcal{G}}$, the equation, $r\left(  x\right)
=r\left(  y\right)  $, holds, it follows that $X$ is also an invariant subset
of $H_{\mathcal{G}}^{\left(  0\right)  }$. Let $\mathfrak{G}_{\mathcal{G}}$ be the reduction of  $H_{\mathcal{G}}$ to $X$, i.e. let $\mathfrak{G}_{\mathcal{G}} = H_{\mathcal{G}}{\vert_X}$. 
(Compare with the first paragraph
following the proof of Theorem 1 in \cite{APP}.) Then
$\mathfrak{G}_{\mathcal{G}}$ is a closed subgroupoid of $H_{\mathcal{G}}$, and
is an $r$-discrete groupoid with counting measures giving a left Haar system.  We will call $\mathfrak
{G}_{\mathcal{G}}$ the \textit{ultrapath groupoid of} $\mathcal{G}$.

For $\left(  x,y\right)  \in S_{\mathcal{G}}$, we define, $\mathcal{A}%
^{\prime}\left(  x,y\right)  =\mathcal{A}\left(  x,y\right)  \cap\mathfrak
{G}_{\mathcal{G}}$, $\mathcal{A}^{\prime}\left(  x,x\right)  =D_{\left(
x,x\right)  }\cap X$, and $\mathcal{A}^{\prime}\left(  \omega\right)
=\left\{  \omega\right\}  \cap X=\emptyset$. Then each $\mathcal{A}^{\prime
}\left(  s\right)  $, $s\in S_{\mathcal{G}}$, is a compact as well as open
subset of $\mathfrak{G}_{\mathcal{G}}$.

Recall that $C^{\ast}\left(  \mathfrak{G}_{\mathcal{G}}\right)  $ is the
completion of the space $C_{c}\left(  \mathfrak{G}_{\mathcal{G}}\right)  $ of
all continuous complex-valued functions on $\mathfrak{G}_{\mathcal{G}}$ with
compact support, with respect to the norm $\left\vert \left\vert f\right\vert
\right\vert =\sup_{\pi}\left\vert \left\vert \pi\left(  f\right)  \right\vert
\right\vert $, where the supremum is taken over all $I$-norm continuous
representations $\pi$ of $C_{c}\left(  \mathfrak{G}_{\mathcal{G}}\right)  $,
(see \cite[p.101]{APB} and \cite[Definition1.5]{Rena}.) For each
$A\in\mathcal{G}^{0}$, let $q_{A}:=1_{\mathcal{A}^{\prime}\left(  A,A\right)
}$ and for $e\in\mathcal{G}^{1}$, let $t_{e}:=1_{\mathcal{A}^{\prime}\left(
\left(  e,r\left(  e\right)  \right)  ,r\left(  e\right)  \right)  }$. We will
show that the family $\left\{  t_{e},q_{A}:e\in\mathcal{G}^{1},A\in
\mathcal{G}^{0}\right\}  $, of characteristic functions on $\mathfrak
{G}_{\mathcal{G}}$ is a Cuntz-Krieger $\mathcal{G}$-family in the groupoid
$C^{\ast}$-algebra, $C^{\ast}\left(  \mathfrak{G}_{\mathcal{G}}\right)  $.
Before we prove our assertion, we clarify our calculations via the following
lemma whose proof we leave to the reader.

\begin{lemma}
\label{lemma 4.9}Let $A$, $B\in\mathcal{G}^{0}$. Observe that $\mathcal{A}%
^{\prime}\left(  A,A\right)  =\{\mu\in X:s\left(  \mu\right)  \subseteq
A\}=\{\mu\in X:A\in\widetilde{s(\mu)}\}$ and, $\mathcal{A}^{\prime}\left(
e,e\right)  =\{ e\cdot\mu\in X:\mu\in\mathfrak{p}\cup\mathfrak{p}^{\infty}\}$.
(See Notation \ref{UltrapathAlgebra}) Then:

\begin{enumerate}
\item [(1)]$\mathcal{A}^{\prime}\left(  A\cap B,A\cap B\right)  =\mathcal{A}%
^{\prime}\left(  A,A\right)  \cap\mathcal{A}^{\prime}\left(  B,B\right)  $;

\item[(2)] $\mathcal{A}^{\prime}\left(  A\cup B,A\cup B\right)  =\mathcal{A}%
^{\prime}\left(  A,A\right)  \cup\mathcal{A}^{\prime}\left(  B,B\right)  $;

\item[(3)] $\mathcal{A}^{\prime}\left(  A,A\right)  =\bigcup\limits_{s\left(
e\right)  \in A}\mathcal{A}^{\prime}\left(  e,e\right)  \cup G^{\prime}\left(
A\right)  $, where $G^{\prime}\left(  A\right)  :=\{ B\in X:B\subseteq A\}$.
\end{enumerate}
\end{lemma}

To check assertion \textit{(2)} of \ Lemma \ref{lemma 4.9}, the reader may use
Theorem V, Page 117 in \cite{Samuel} and the fact that a set $B\in
\mathcal{G}^{0}$ that lies in $X$ is such that its principal filter,
$\widetilde{B}$, is an ultrafilter over $G^{0}$.

\begin{proposition}
\label{Proposition 4.11}The family $\left\{  t_{e},q_{A}:A\in\mathcal{G}%
^{0},e\in\mathcal{G}^{1}\right\}  $ is a Cuntz-Krieger $\mathcal{G}$-family in
$C^{\ast}\left(  \mathfrak{G}_{\mathcal{G}}\right)  $.
\end{proposition}

\begin{proof}
We leave to the reader to check that the $t_{e}$'s are partial isometries with
mutually orthogonal ranges and the $q_{A}$'s are projections. Also it is easy
to verify that $t_{e}^{\ast}t_{e}=q_{r\left(  e\right)  }$ and $t_{e}%
t_{e}^{\ast}\leq q_{s\left(  e\right)  }$. Then we have $q_{\emptyset
}=1_{\mathcal{A}^{\prime}\left(  \emptyset,\emptyset\right)  }=1_{\emptyset
}=0$. Also for all $A$, $B\in\mathcal{G}^{0}$, we see by (1) and (2) in Lemma
\ref{lemma 4.9} that the equations, $q_{A\cap B}=q_{A}q_{B}$ and $q_{A\cup
B}=$ $q_{A}+q_{B}-q_{A\cap B}$ hold. (Recall Remark \ref{principal
difficulty}.) Finally let $v\in G^{0}$ such that $0<|s^{-1}\left(  v\right)
|<\infty$. Then by (3) of Lemma \ref{lemma 4.9}, $q_{v}-\sum\limits_{s\left(
e\right)  =v}t_{e}t_{e}^{\ast}=1_{\mathcal{A}^{\prime}\left(  v,v\right)
}-\sum\limits_{s\left(  e\right)  =v}1_{\mathcal{A}^{\prime}(e,e)}%
=1_{G^{\prime}\left(  \{ v\}\right)  }$. But since $0<|s^{-1}\left(  v\right)
|<\infty$, it follows that, $\left\{  v\right\}  \notin X$. Consequently,
$G^{\prime}\left(  \left\{  v\right\}  \right)  =\emptyset$, and hence
$q_{v}-\sum\limits_{s\left(  e\right)  =v}t_{e}t_{e}^{\ast}=0$.
\end{proof}

The following straightforward lemma tells us that the collection of compact
open subsets of $\mathfrak{G}_{\mathcal{G}}$, $\left\{  \mathcal{A}^{\prime
}\left(  x,y\right)  :x,y\in S_{\mathcal{G}}\right\}  $, is a subbasis for the
topology of $\mathfrak{G}_{\mathcal{G}}$.

\begin{lemma}
\label{subbasis} Given $\left(  x,y\right)  $ and $\left(  z,w\right)  $ in
$S_{\mathcal{G}}$, then
\[
\mathcal{A}^{\prime}\left(  x,y\right)  \cap\mathcal{A}^{\prime}\left(
z,w\right)  =\left\{
\begin{array}
[c]{cc}%
\mathcal{A}^{\prime}\left(  x,y\right)  & \text{if }x=z\cdot x^{\prime},\text{
}y=w\cdot x^{\prime},\text{ for some }x^{\prime}\in\mathfrak{p}\text{;}\\
\mathcal{A}^{\prime}\left(  z,w\right)  & \text{if }z=x\cdot z^{\prime},\text{
}w=y\cdot z^{\prime},\text{ for some }z^{\prime}\in\mathfrak{p}\text{;}\\
\mathcal{A}^{\prime}\left(  z,z\right)  & \text{if }x=y\in\mathcal{G}%
^{0},z=w,\text{ }s\left(  z\right)  \in x,|z|\geq1\text{;}\\
\emptyset & \text{otherwise}%
\end{array}
\right.
\]
\end{lemma}

Next define $\Lambda^{\prime}:S_{\mathcal{G}}\longrightarrow\mathfrak
{G}_{\mathcal{G}}^{a}$, by setting $\Lambda^{\prime}\left(  x,y\right)
=\mathcal{A}^{\prime}\left(  x,y\right)  $. Since $X$ is closed and invariant
subset of $H_{\mathcal{G}}^{\left(  0\right)  }$, $\Lambda^{\prime}$ is a well
defined inverse semigroup isomorphism from $S_{\mathcal{G}}$ into the ample
inverse semigroup $\mathfrak{G}_{\mathcal{G}}^{a}$. Then $\Lambda^{\prime
}\left(  S_{\mathcal{G}}\right)  $ is an inverse subsemigroup of $\mathfrak
{G}_{\mathcal{G}}^{a}$ which is a subbasis for the topology of $\mathfrak
{G}_{\mathcal{G}}$ by Lemma \ref{subbasis}. In fact the span $W$ of
characteristic functions $1_{\mathcal{A}^{\prime}\left(  x,y\right)  }$ for
$\left(  x,y\right)  \in S_{\mathcal{G}}$, is $I$-norm dense in $C_{c}\left(
\mathfrak{G}_{\mathcal{G}}\right)  $, \cite[Proposition 2.2.7]{APB}. Let
$\left\{  s_{e},p_{A}:e\in\mathcal{G}^{1},A\in\mathcal{G}^{0}\right\}  $ be a
universal Cuntz-Krieger $\mathcal{G}$-family in $C^{\ast}\left(
\mathcal{G}\right)  $ and define a map $\psi$ on the set of generators of
$C^{\ast}\left(  \mathcal{G}\right)  $ into $C^{\ast}\left(  \mathfrak
{G}_{\mathcal{G}}\right)  $, by the equations:

\begin{enumerate}
\item $\psi\left(  s_{\alpha}p_{A}s_{\beta}^{\ast}\right)  :=1_{\mathcal{A}%
^{\prime}(x,y)}$, where $x=\left(  \alpha,r\left(  \alpha\right)  \cap
r\left(  \beta\right)  \cap A\right)  $ and $y$ = $(\beta$, $r\left(
\alpha\right)  \cap r\left(  \beta\right)  \cap A)$;

\item $\psi\left(  s_{\alpha}p_{A}\right)  :=1_{\mathcal{A}^{\prime}\left(
x,r\left(  x\right)  \right)  }$, where $x=\left(  \alpha,r\left(
\alpha\right)  \cap A\right)  $; and

\item $\psi\left(  p_{A}\right)  :=1_{\mathcal{A}^{\prime}\left(  A,A\right)
}$, $A\in\mathcal{G}^{0}$.
\end{enumerate}

Observe that $s_{\alpha}p_{A}s_{\beta}^{\ast}\neq0$ precisely when $r\left(
\alpha\right)  \cap r\left(  \beta\right)  \cap A\neq\emptyset$. Then $\psi$
extends to a surjective homomorphism which we shall denote also by $\psi$,
such that $\psi\left(  s_{e}\right)  =1_{\mathcal{A}^{\prime}\left(  \left(
e,r\left(  e\right)  \right)  ,r\left(  e\right)  \right)  }$ and $\psi\left(
p_{A}\right)  =1_{\mathcal{A}^{\prime}\left(  A,A\right)  }$. Moreover, since
we are assuming that $\mathcal{G}$ has no sinks, the inequality, $\psi\left(
p_{A}\right)  \neq0$, holds for all nonempty sets $A$ in $\mathcal{G}^{0}$.
Let $\gamma$ be the gauge action for $C^{\ast}\left(  \mathcal{G}\right)  $,
see \cite[p.7]{MTT}. For $z\in$ $\mathbb{T}$ define $\tau_{z}:C^{\ast}\left(
\mathfrak{G}_{\mathcal{G}}\right)  \longrightarrow C^{\ast}\left(  \mathfrak
{G}_{\mathcal{G}}\right)  $, by the equation, $\tau_{z}\left(  f\right)
\left(  h\right)  =z^{c\left(  h\right)  }f\left(  h\right)  $, where $f\in
C_{c}\left(  \mathfrak{G}_{\mathcal{G}}\right)  $, $h\in\mathfrak
{G}_{\mathcal{G}}$ and $c:\mathfrak{G}_{\mathcal{G}}\longrightarrow\mathbb{Z}%
$, is the cocycle defined by $c\left(  \left(  x,k,x^{\prime}\right)  \right)
=k$, $k\in\mathbb{Z}$. Notice that $\tau$ is a strongly continuous action of
$\mathbb{T}$ on $C^{\ast}\left(  \mathfrak{G}_{\mathcal{G}}\right)  $. Also by
Proposition \ref{Proposition 4.11} the collection, $\left\{  1_{\mathcal{A}%
^{\prime}\left(  A,A\right)  },1_{\mathcal{A}^{\prime}\left(  \left(
e,r\left(  e\right)  \right)  ,r\left(  e\right)  \right)  }:A\in
\mathcal{G}^{0},e\in\mathcal{G}^{1}\right\}  $, is a Cuntz-Krieger
$\mathcal{G}$- family in $C^{\ast}\left(  \mathfrak
{G}_{\mathcal{G}}\right)  $. A simple computation shows that, $\psi\circ
\gamma_{z}\left(  s_{e}\right)  =\tau_{z}\circ\psi\left(  s_{e}\right)  $ for
all $e\in\mathcal{G}^{1}$, and $\psi\circ\gamma_{z}\left(  p_{A}\right)
=\tau_{z}\circ\psi\left(  p_{A}\right)  $, for all $A\in\mathcal{G}^{0}$. Thus
the equation $\psi\circ\gamma_{z}=\tau_{z}\circ\psi$, holds for all $z\in$
$\mathbb{T}$. By the Gauge-Invariant Uniqueness Theorem for ultragraphs,
\cite[Theorem 6.8]{MTT}, $\psi$ is faithful and hence an isomorphism from
$C^{\ast}\left(  \mathcal{G}\right)  $ onto $C^{\ast}\left(  \mathfrak
{G}_{\mathcal{G}}\right)  $.

We may thus summarize our analysis to this point in the following theorem, which is a corollary to
Theorem \ref{nat.correspondence}.

\begin{theorem}\label{Isomorphism}
If $\mathcal{G}$ is an ultragraph without sinks, then $C^{\ast}\left(  \mathcal{G}\right)  \simeq C_{0}^{\ast}\left(  S_{\mathcal{G}%
}\right)  \simeq C^{\ast}\left(  \mathfrak{G}_{\mathcal{G}}\right)  $.
\end{theorem}

\section{{\protect\small ULTRAGRAPH GROUPOIDS ARE AMENABLE}}

Let $G$ be a discrete group, let $\mathcal{G}=\left(  G^{0},\mathcal{G}%
^{1},r,s\right)  $ be an ultragraph and let $\varphi:\mathcal{G}%
^{1}\longrightarrow$ $G$ be a function. We introduce an analog of the skew
product graph considered in \cite{kp}; the resulting object, which we denote
by $\mathcal{G}\times_{\varphi}G$, is also an ultragraph. We show that the
crossed product of $C^{\ast}\left(  \mathcal{G}\right)  $ by the gauge action,
$C^{\ast}\left(  \mathcal{G}\right)  \rtimes_{\gamma}\mathbb{T}$, is
isomorphic to $C^{\ast}\left(  \mathcal{G}\times_{\varphi}\mathbb{Z}\right)
$, where $\varphi$ is the constant function $1$ on $\mathcal{G}^{1}$. In this
case, we shall write $\mathcal{G}\times_{\varphi}\mathbb{Z}$ as $\mathcal{G}%
\times_{1}\mathbb{Z}$ . It turns out that the ultragraph $\mathcal{G}%
\times_{1}\mathbb{Z}$ has no loops and so by Theorem 4.1 in \cite{MTPS},
$C^{\ast}\left(  \mathcal{G}\times_{1}\mathbb{Z}\right)  $ is an AF-algebra.
It will then follow that $C^{\ast}\left(  \mathcal{G}\right)  \rtimes_{\gamma
}\mathbb{T}$ is AF and, consequently, that $\mathfrak{G}_{\mathcal{G}}$ is amenable.

\begin{definition}
Let $G$ be a discrete group and let $\mathcal{G}=\left(  G^{0},\mathcal{G}%
^{1},r,s\right)  $ an ultragraph. Given a function $\varphi:\mathcal{G}%
^{1}\longrightarrow G$ then the \textbf{skew product ultragraph}
$\mathcal{G}\times_{\varphi}G$ is defined as follows: the set of vertices is
$G^{0}\times G$, the set of edges is $\mathcal{G}^{1}\times G$ and the
structure maps $s^{\prime}:\mathcal{G}^{1}\times G\longrightarrow
\mathcal{G}^{0}\times G$ and $r^{\prime}:\mathcal{G}^{1}\times
G\longrightarrow P\left(  \mathcal{G}^{0}\times G\right)  $, are defined by
the equations,
\[
s^{\prime}\left(  e,g\right)  =\left(  s\left(  e\right)  ,g\right)  \text{
and }r^{\prime}\left(  e,g\right)  =r\left(  e\right)  \times\left\{
g\varphi\left(  e\right)  \right\}  \text{.}%
\]
We write $\mathcal{G}\times_{\varphi}G$ for $\left(  G^{0}\times
G,\mathcal{G}^{1}\times G,r^{\prime},s^{\prime}\right)  $. 
\end{definition}

It is clear that
$\mathcal{G}\times_{\varphi}G$ is an ultragraph.

\begin{remark}
\label{singular}We note that the ultragraph $\mathcal{G}$ has no singular
vertices if and only if the skew ultragraph $\mathcal{G}\times_{\varphi}G$ has
no singular vertices.
\end{remark}

\begin{proof}
This follows from the fact that, for any $v\in G^{0}$ and any $g\in G$,
$\left(  s^{\prime}\right)  ^{-1}\left(  v,g\right)  =s^{-1}\left(  v\right)
\times\left\{  g\right\}  $.
\end{proof}

We may assume that $\mathcal{G}$ has no singular vertices, i.e. no vertices
which are infinite emitters. The reason for this is that given an ultragraph
$\mathcal{G}$, one may build a new ultragraph $\mathcal{F}$ that has no
singular vertices \cite[Prop.6.2, p.17]{MTT} such that $C^{\ast}\left(
\mathcal{G}\right)  $ is strongly Morita equivalent to $C^{\ast}\left(
\mathcal{F}\right)  $. ($\mathcal{F}$ is called the \emph{desingularization}
of $\mathcal{G}$.) Further, since $C^{\ast}\left(  \mathcal{F}\right)  $ is
AF, as we shall see, and since the property of being AF is preserved under
strong Morita equivalence, we may conclude that $C^{\ast}\left(
\mathcal{G}\right)  $ is AF. See also \cite[Proposition 6.6.]{MTT}.

\begin{proposition}
If $\mathcal{G}$ is an ultragraph with no singular vertices, then the unit
space $\mathfrak{G}_{\mathcal{G}}^{\left(  0\right)  }$ of its groupoid model
$\mathfrak{G}_{\mathcal{G}}$ becomes $\mathfrak{p}^{\infty}$, where
$\mathfrak{p}^{\infty}$ denotes the infinite path space of $\mathcal{G}$.
\end{proposition}

\begin{proof}
The unit space, $\mathfrak{G}_{\mathcal{G}}^{\left(  0\right)  }$, of
$\mathfrak{G}_{\mathcal{G}}$ is the set $Y_{\infty}\cup\mathfrak{p}^{\infty
}=\overline{\mathfrak{p}^{\infty}}$. Take any $y$ in $Y_{\infty}$. If
$0<\left|  y\right|  $, then the set, $\mathcal{A}^{\prime}\left(  \left\{
s\left(  y\right)  \},\{ s\left(  y\right)  \right\}  \right)  $, is an open
subset of $\mathfrak{G}_{\mathcal{G}}^{\left(  0\right)  }$ which contains
$y$. Thus there is an infinite sequence of infinite paths $\gamma_{k}$ in
$\mathfrak
{p}^{\infty}$ such that $s\left(  \gamma_{k}\right)  =s\left(  y\right)  $ for
large $k$. Thus, the vertex, $s\left(  y\right)  $, is an infinite emitter and
hence a singular vertex in $\mathcal{G}$, contrary to hypothesis. If $\left|
y\right|  =0$, say $y=A$ in $\mathcal{G}^{0}$, then set $\varepsilon\left(
A\right)  $ is infinite and $\widetilde{A}$ is an ultrafilter over $G^{0}$.
(See Proposition \ref{closure}) Therefore the set $A$ must be infinite.
(Otherwise $\mathcal{G}$ would have a singular vertex as well.) By Lemma 2.12
in \cite{MTT}, there are finite subsets $Y_{1},\ldots,Y_{n}$ of $\mathcal{G}%
^{1}$ and a finite subset $F$ of $G^{0}$ such that $A=\bigcap\limits_{e\in
Y_{1}}r(e)\cup\ldots\cup\bigcap\limits_{e\in Y_{n}}r(e)\cup F$. Furthermore,
$F$ may be chosen to be disjoint from $\bigcap\limits_{e\in Y_{1}}%
r(e)\cup\ldots\cup\bigcap\limits_{e\in Y_{n}}r(e)$. So there is a finite
number of edges $e_{1},\ldots,e_{n}$ in the sets $Y_{1},\ldots,Y_{n}$,
respectively, such that $A\subseteq r\left(  e_{1}\right)  \cup\ldots\cup
r\left(  e_{n}\right)  $. Therefore the set, $r\left(  e_{1}\right)
\cup\ldots\cup r\left(  e_{n}\right)  $, lies in the ultrafilter,
$\widetilde{A}$. Hence by Theorem V in \cite{Samuel}, we have, $r\left(
e_{i}\right)  \in\widetilde{A}$ for some $i\in\left\{  1,\ldots,n\right\}  $.
Thus the inclusion, $A\subseteq r\left(  e_{i}\right)  $, holds. But then, the
ultrapath, $\left(  e_{i},A\right)  $, would be in $Y_{\infty}$. Then, as in
the case when the length of $y$, $\left|  y\right|  $, is positive, the
vertex, $s(e_{i})$, will contradict the condition that $\mathcal{G}$ has no
singular vertices. Therefore $Y_{\infty}=\emptyset$ and hence $\mathfrak
{G}_{\mathcal{G}}^{\left(  0\right)  }=\mathfrak{p}^{\infty}$.
\end{proof}

Recall that the position cocycle $c:\mathfrak{G}_{\mathcal{G}}\longrightarrow
\mathbb{Z}$ is given by the formula $c\left(  \chi,k,\chi^{\prime}\right)
=k$, $k\in\mathbb{Z}$. In the following theorem we show that the skew product
groupoid obtained from $c$, $\mathfrak{G}_{\mathcal{G}}\times_{c}\mathbb{Z}$,
(see \cite{Rena}) is the same as the path groupoid $\mathfrak{G}%
_{\mathcal{G}\times_{1}\mathbb{Z}}$ of the skew ultragraph $\mathcal{G}%
\times_{1}\mathbb{Z}$.

\begin{theorem}
\label{skiso}Let $\mathcal{G}=\left(  G^{0},\mathcal{G}^{1},r,s\right)  $ be
an ultragraph with no singular vertices. Let $G=\left(  \mathbb{Z},+\right)  $
be the discrete group of the integers under addition and let $\varphi
:\mathcal{G}^{1}\longrightarrow\mathbb{Z}$ be the function defined by
$\varphi\left(  e\right)  =1$ for $e\in\mathcal{G}^{1}$. Then the groupoid
model $\mathfrak{G}_{\mathcal{G}\times_{1}\mathbb{Z}}$ for $\mathcal{G}%
\times_{1}\mathbb{Z}$ \ is isomorphic to the skew product groupoid $\mathfrak
{G}_{\mathcal{G}}\times_{c}\mathbb{Z}$, where $c:\mathfrak{G}_{\mathcal{G}%
}\longrightarrow\mathbb{Z}$ is the position cocycle on $\mathfrak
{G}_{\mathcal{G}}$.
\end{theorem}

\begin{proof}
Since $\mathcal{G}$ has no singular vertices, we may identify the unit space
of $\mathfrak{G}_{\mathcal{G}}$ with $\mathfrak{p}^{\infty}$ by Proposition
\ref{closure}. Next, we identify the unit space $\mathfrak{p}_{\mathcal{G}%
\times_{1}\mathbb{Z}}^{\infty}$ of $\mathfrak{G}_{\mathcal{G}\times
_{1}\mathbb{Z}}$ with the unit space $\mathfrak{p}^{\infty}\times\mathbb{Z}$
of $\mathfrak{G}_{\mathcal{G}}\times_{c}\mathbb{Z}$ as follows: for $\left(
\gamma,m\right)  \in\mathfrak{p}^{\infty}\times\mathbb{Z}$, define
$f:\mathfrak{p}^{\infty}\times\mathbb{Z}\longrightarrow\mathfrak
{p}_{\mathcal{G}\times_{1}\mathbb{Z}}^{\infty}$ by
\[
f\left(  \gamma,m\right)  :=\left(  e_{1},m\right)  \left(  e_{2},m+1\right)
\ldots\text{,}%
\]
for $\gamma=e_{1}e_{2}\ldots$. It is straightforward to check that this
defines an infinite path in $\mathcal{G}\times_{1}\mathbb{Z}$. Define a shift
$\sigma:\mathfrak{p}^{\infty}\longrightarrow\mathfrak{p}^{\infty}$ by the
formula,
\[
\sigma\left(  \gamma\right)  =e_{2}e_{3}\ldots\text{, }\gamma=e_{1}e_{2}%
e_{3}\ldots\in\mathfrak{p}^{\infty}\text{,}%
\]
and define another shift $\widetilde{\sigma}:\mathfrak{p}^{\infty}%
\times\mathbb{Z}\longrightarrow\mathfrak{p}^{\infty}\times\mathbb{Z}$ by the
formula, $\widetilde{\sigma}\left(  \gamma,n\right)  =\left(  \sigma\left(
\gamma\right)  ,n+1\right)  $. Under this identification the groupoid model
$\mathfrak{G}_{\mathcal{G}}$ for $\mathcal{G}$ is given by the equation,
\[
\mathfrak{G}_{\mathcal{G}}=\{\left(  \gamma,m-l,\gamma^{\prime}\right)
:\gamma,\gamma^{\prime}\in\mathfrak{p}^{\infty},\sigma^{m}\left(
\gamma\right)  =\sigma^{l}\left(  \gamma^{\prime}\right)  \}\text{,}%
\]
while the groupoid model $\mathfrak{G}_{\mathcal{G}\times_{1}\mathbb{Z}}$ for
$\mathcal{G}\times_{1}\mathbb{Z}$ is given by the equation,
\[
\mathfrak{G}_{\mathcal{G}\times_{1}\mathbb{Z}}=\{\left(  \left(
\gamma,n\right)  ,m-l,\left(  \gamma^{\prime},k\right)  \right)
:\gamma,\gamma^{\prime}\in\mathfrak{p}^{\infty},\widetilde{\sigma}^{m}\left(
\gamma,n\right)  =\widetilde{\sigma}^{l}\left(  \gamma^{\prime},k\right)
,n,k\in\mathbb{Z}\}\text{.}%
\]
Define a map $\phi:\mathfrak{G}_{\mathcal{G}}\times_{1}\mathbb{Z}%
\longrightarrow\mathfrak{G}_{\mathcal{G}\times_{1}\mathbb{Z}}$ as follows: for
$\gamma$ and $\gamma^{\prime}\in\mathfrak{p}^{\infty}$ with $\sigma^{m}\left(
\gamma\right)  =\sigma^{l}\left(  \gamma^{\prime}\right)  $ and $n\in
\mathbb{Z}$, set
\[
\phi\left[  (\left(  \gamma,m-l,\gamma^{\prime}\right)  ,n)\right]  :=\left(
\left(  \gamma,n\right)  ,m-l,\left(  \gamma^{\prime},n+m-l\right)  \right)
\text{.}%
\]
Note that
\begin{align*}
\widetilde{\sigma}^{m}\left(  \gamma,n\right)   &  =\left(  \sigma^{m}\left(
\gamma\right)  ,n+m\right) \\
&  =\left(  \sigma^{l}\left(  \gamma^{\prime}\right)  ,n+m\right) \\
&  =\left(  \sigma^{l}\left(  \gamma^{\prime}\right)  ,(n+m-l)+l\right) \\
&  =\widetilde{\sigma}^{l}\left(  \gamma^{\prime},n+m-l\right)  \text{,}%
\end{align*}
hence $\left(  \left(  \gamma,n\right)  ,m-l,\left(  \gamma^{\prime
},n+m-l\right)  \right)  \in\mathfrak{G}_{\mathcal{G}\times_{1}\mathbb{Z}}$.
So $\phi$ is well defined. The rest of the proof proceeds as in \cite[Theorem
2.4]{kp}.
\end{proof}

In order to show that $C^{\ast}\left(  \mathcal{G}\right)  \rtimes_{\gamma
}\mathbb{T}$ is AF, we need the following lemma.

\begin{lemma}
\label{af} Let $\mathcal{G}=\left(  G^{0},\mathcal{G}^{1},r,s\right)  $ be an
ultragraph. Let $G=\left(  \mathbb{Z},+\right)  $ be the discrete group of the
integers under addition and let $\varphi:\mathcal{G}^{1}\longrightarrow
\mathbb{Z}$ be the function defined by $\varphi\left(  e\right)  =1$, for all
$e\in\mathcal{G}^{1}$. Then the ultragraph $C^{\ast}$-algebra, $C^{\ast
}\left(  \mathcal{G}\times_{1}\mathbb{Z}\right)  $ is an AF-algebra.
\end{lemma}

\begin{proof}
Observe that the ultragraph, $\mathcal{G}\times_{1}\mathbb{Z}$, has no loops.
Thus by Theorem 4.1, in \cite{MTPS}, $C^{\ast}\left(  \mathcal{G}\times
_{1}\mathbb{Z}\right)  $ is an AF-algebra.
\end{proof}

\begin{corollary}
\label{amena} Let $\mathcal{G}=\left(  G^{0},\mathcal{G}^{1},r,s\right)  $ be
an ultragraph, with no singular vertices. Then the groupoid $C^{\ast}%
$-algebra, $C^{\ast}\left(  \mathfrak{G}_{\mathcal{G}}\times_{c}%
\mathbb{Z}\right)  $ is an AF-algebra. Furthermore, $C^{\ast}\left(
\mathfrak{G}_{\mathcal{G}}\times_{c}\mathbb{Z}\right)  $ is nuclear and hence
$\mathfrak{G}_{\mathcal{G}}\times_{c}\mathbb{Z}$ is amenable.
\end{corollary}

\begin{proof}
The proof follows from Theorem \ref{skiso} and Lemma \ref{af}.
\end{proof}

\begin{theorem}
If $\mathcal{G}$ is an ultragraph with no singular vertices, then $C^{\ast
}\left(  \mathcal{G}\right)  \rtimes_{\gamma}\mathbb{T}$ is AF and the
groupoid $\mathfrak{G}_{\mathcal{G}}$ is amenable.
\end{theorem}

\begin{proof}
Fix $z\in\mathbb{T}$ and let $\alpha_{z}\left(  f\right)  \left(  h\right)
=z^{c\left(  h\right)  }f\left(  h\right)  $, for $f\in C_{c}\left(  \mathfrak
{G}_{\mathcal{G}}\right)  $, $h\in\mathfrak{G}_{\mathcal{G}}$. Then
$\alpha_{z}\in\mathrm{Aut}C^{\ast}\left(  \mathfrak{G}_{\mathcal{G}}\right)  $
and $\{\alpha_{z}\}_{z\in\mathbb{T}}$ is a strongly continuous action of
$\mathbb{T}$ on $C^{\ast}(\mathfrak{G}_{\mathcal{G}})$ (see \cite[Proposition
5.1, p.110]{Rena}). So we can form the crossed product $C^{\ast}$-algebra
$C^{\ast}\left(  \mathfrak{G}_{\mathcal{G}}\right)  \rtimes_{\alpha}%
\mathbb{T}$, and by Theorem 5.7 in \cite[p.118]{Rena}, we have $C^{\ast
}\left(  \mathfrak
{G}_{\mathcal{G}}\right)  \rtimes_{\alpha}\mathbb{T}{\simeq}C^{\ast}\left(
\mathfrak{G}_{\mathcal{G}}\times_{c}\mathbb{Z}\right)  $. Recall the gauge
action $\gamma$ of $\mathbb{T}$ on $C^{\ast}\left(  \mathcal{G}\right)  $.
Since $C^{\ast}\left(  \mathcal{G}\right)  $ is defined to be the universal
$C^{\ast}$-algebra generated by the $s_{e}$, $p_{A}$'s subject to the
relations in Definition \ref{cuntz-pimsner rep.} and the gauge action on
$C^{\ast}\left(  \mathcal{G}\right)  $ preserves these relations (and so does
$\alpha$ via $s_{e}\longrightarrow1_{\mathcal{A}^{\prime}\left(  \left(
e,r\left(  e\right)  \right)  ,r\left(  e\right)  \right)  }$ and
$p_{A}\longrightarrow1_{\mathcal{A}^{\prime}\left(  A,A\right)  }$), we have
\begin{align*}
C^{\ast}\left(  \mathcal{G}\right)  \rtimes_{\gamma}\mathbb{T}  &  \simeq
C^{\ast}\left(  \mathfrak{G}_{\mathcal{G}}\right)  \rtimes_{\alpha}%
\mathbb{T}\\
&  \simeq C^{\ast}\left(  \mathfrak{G}_{\mathcal{G}}\times_{c}\mathbb{Z}%
\right) \\
&  \simeq C^{\ast}\left(  \mathfrak{G}_{\mathcal{G}\times_{1}\mathbb{Z}%
}\right)  \simeq C^{\ast}\left(  \mathcal{G}\times_{1}\mathbb{Z}\right)
\text{.}%
\end{align*}
Thus $C^{\ast}\left(  \mathcal{G}\right)  \rtimes_{\gamma}\mathbb{T}$ is an
AF-algebra. Corollary \ref{amena} implies that 
$\mathfrak{G}_{\mathcal{G}}\times_{c}\mathbb{Z}$ is amenable. Since $\mathbb{Z}$ is amenable we may apply \cite[Proposition
II.3.8]{Rena} to deduce that $\mathfrak{G}_{\mathcal{G}}$ is amenable.
\end{proof}

As we mentioned earlier, we can extend our result to general ultragraphs using desingularization.

\begin{theorem}
All ultragraph groupoids are amenable.
\end{theorem}

\begin{proof}
Let $\mathcal{G}=\left(  G^{0},\mathcal{G}^{1},r,s\right)  $ be an ultragraph,
and let $\mathcal{F}$ be a desingularization of $\mathcal{G}$. Then
$\mathcal{F}$ is an ultragraph with no singular vertices. Thus the groupoid
$\mathfrak{G}_{\mathcal{F}}$ is amenable. But then we have $C^{\ast}\left(
\mathfrak{G}_{\mathcal{G}}\right)  \simeq C^{\ast}\left(  \mathcal{G}\right)
$ and $C^{\ast}\left(  \mathcal{F}\right)  \simeq C^{\ast}(\mathfrak
{G}_{\mathcal{F}})$. By Theorem 6.6 in \cite{MTT} $C^{\ast}\left(
\mathcal{F}\right)  $ is strongly Morita equivalent to $C^{\ast}\left(
\mathcal{G}\right)  $. Thus $C^{\ast}\left(  \mathfrak{G}_{\mathcal{G}%
}\right)  $ is strongly Morita equivalent to $C^{\ast}(\mathfrak
{G}_{\mathcal{F}})$. Therefore $\mathfrak{G}_{\mathcal{G}}$ is amenable.
\end{proof}

\section{{\protect\small THE SIMPLICITY OF $C^{\ast}\left(  \mathcal{G}%
\right)  $}}

In this section we will use the groupoid $\mathfrak{G}_{\mathcal{G}}$ to
obtain conditions sufficient for $C^{\ast}\left(  \mathcal{G}\right)  $ to be
simple. The result obtained is effectively due to Mark Tomforde who also
proved the converse (see \cite[Theorem 3.11]{MTPS}). It seems likely that the
approach to the converse can be adapted to apply within the groupoid context
of the present paper.

For a vertex $v\in G^{0}$, \textit{a loop based at} $v$ is a finite path
$\alpha=e_{1}\ldots e_{n}$ in $\mathcal{G}$, such that $s\left(
\alpha\right)  =v$, $v\in r\left(  \alpha\right)  $ and $v\neq s\left(
e_{i}\right)  $ for all $1<i\leq n$. When $\alpha$ is a loop based at $v$, we
say that $v$ \textit{hosts the loop} $\alpha$. A loop based at $v$ may pass
through other vertices $w\neq v$ more than once but no edge other than $e_{1}$
may have source $v$. The ultragraph $\mathcal{G}$ is said to satisfy
\textit{condition (K)} if every $v\in G^{0}$ which hosts a loop hosts at least
two distinct loops. (See \cite[Defintion 7.1, p.17,18]{KMST}.)

Recall that a locally compact groupoid $G$ is \textit{essentially principal}
(\cite[p.100]{Rena}) if for all nonempty closed invariant subset $F$ of
$\ $its unit space, $G^{\left(  0\right)  }$ the set, $\{ x\in F:x$ has
trivial isotropy$\}$, is dense in $F$. We now show that for a general
ultragraph $\mathcal{G}$, the groupoid $\mathfrak{G}_{\mathcal{G}}$ is
essentially principal if and only if $\mathcal{G}$ satisfies condition (K).

\begin{theorem}
\label{essp} If $\mathcal{G}=\left(  G^{0},\mathcal{G}^{1},r,s\right)  $ is an
ultragraph, the $r$-discrete groupoid $\mathfrak{G}_{\mathcal{G}}$ is
essentially principal if and only if \ $\mathcal{G}$ satisfies condition (K).
\end{theorem}

\begin{proof}
Suppose that $\mathcal{G}$ satisfies condition (K). Every ultrapath in
$\mathfrak{p}$ has trivial isotropy. So we just need to consider the infinite
paths. Let $F$ be a nonempty closed invariant subset of $\mathfrak{p}^{\infty
}$. We have to show that the set of points in $F$ with trivial isotropy is
dense in $F$. So fix any $\chi\in F$, and fix a basic open neighborhood
$D_{\left(  x,x\right)  }\cap F$ of $\chi$, ($x\in\mathfrak{p}$, $\left|
x\right|  \geq1$). Note that $\chi$ must have the form $x\cdot\gamma$ where
$\gamma\in\mathfrak{p}^{\infty}$ and $s\left(  \gamma\right)  \in r\left(
x\right)  $. (See Lemma \ref{topoinf}.) If every vertex through which $\gamma$
passes hosts no loop, then $\gamma$ must pass through each of them exactly
once. Thus the triple, $\left(  x\cdot\gamma,k,x\cdot\gamma\right)  $ can
belong to $\mathfrak{G}_{\mathcal{G}}$ only if $k=0$, and $\chi$ itself has
trivial isotropy. So we may assume that the infinite path $\gamma$ passes
through some vertex $v$ hosting a loop, say $\gamma=\beta\cdot\gamma^{\prime}%
$, $\gamma^{\prime}\in\mathfrak{p}^{\infty}$, with $s\left(  \gamma^{\prime
}\right)  =v$. Let $\mu$ and $\nu$ be distinct loops based at $v$, and define
paths $\gamma_{n}\in\mathfrak{p}^{\infty}$ by%
\[
\gamma_{n}:=x\cdot\beta\mu\nu\mu\mu\nu\nu\cdots\overset{n}{\overbrace
{\mu\cdots\mu}}\overset{n}{\overbrace{\nu\cdots\nu}}\cdot\gamma^{\prime
}\text{.}%
\]
Observe that each triple, $\left(  \gamma_{n},k_{n},x\cdot\beta\cdot
\gamma^{\prime}\right)  $, belongs to $\mathfrak{G}_{\mathcal{G}}$, (with a
substantial lag $k_{n}$), and since $F$ is invariant, each $\gamma_{n}$ lies
in $F$. The sequence $\gamma_{n}$ converges to the infinite path,
\[
x\cdot\beta\mu\nu\mu\mu\nu\nu\cdots\overset{n}{\overbrace{\mu\cdots\mu}%
}\overset{n}{\overbrace{\nu\cdots\nu}}\cdots\text{,}%
\]
which has trivial isotropy and belongs to $D_{\left(  x,x\right)  }\cap F$
because $F$ is closed. So we have approximated $\chi$ by a point with trivial
isotropy. Thus the groupoid $\mathfrak{G}_{\mathcal{G}}$ is essentially principal.

Suppose conversely that $\mathfrak{G}_{\mathcal{G}}$ is essentially principal
and suppose that $v\in G^{0}$ is a vertex hosting exactly one loop
$\alpha=e_{1}\ldots e_{n}$. Consider the set,
\[
C=\left\{  \gamma=\gamma_{1}\gamma_{2}\ldots\in\mathfrak{p}^{\infty}:s\left(
\gamma_{i}\right)  \geq v\text{ for all }i\geq1\right\}  \text{.}%
\]
Note that the infinite path $\gamma=\alpha\alpha\ldots$ belongs to $C$. Let
$F=\overline{C}$, the closure of $C$ in $X$. We show $F$ is invariant subset
of $X$. So take any $h=\left(  \chi,k,\chi^{\prime}\right)  \in\mathfrak
{G}_{\mathcal{G}}$ and suppose that $\chi\in F$. Then there is a sequence of
infinite paths, $\gamma_{n}\in C$ such that $\gamma_{n}\longrightarrow\chi$.
Either $|\chi|=\infty$ or $|\chi|<\infty$. Suppose first that $\left|
\chi\right|  =\infty$. Then for some $x,y\in$ $\mathfrak{p}$, and $\mu
\in\mathfrak{p}^{\infty}$, we have $\chi=x\cdot\mu$ and $\chi^{\prime}%
=y\cdot\mu$. Then eventually, every $\gamma_{n}=x\cdot\mu_{n}$, where $\mu
_{n}$ is a sequence in $\mathfrak{p}^{\infty}$ such that $s\left(  \mu
_{n}\right)  \in r\left(  x\right)  $, and so $y\cdot\mu_{n}\longrightarrow
y\cdot\mu$ eventually. Next we show that $y\cdot\mu_{n}\in C$ eventually. It
will follow that $\chi^{\prime}=y\cdot\mu\in F$. For this end, if $\left|
y\right|  =0$, then since $x\cdot\mu_{n}\in C$ eventually, $\mu_{n}$ must be
in $C$ eventually. But since $\left|  y\right|  =0$, the equation, $y\cdot
\mu_{n}=$ $\mu_{n}$, holds. Thus $y\cdot\mu_{n}\in C$ eventually. Otherwise
let $y=\left(  \beta,B\right)  $ for some finite path $\beta$ with positive
length. Set $\beta=\beta_{1}\ldots\beta_{\left|  \beta\right|  }$, and,
$\mu_{n}=\mu_{n1}\mu_{n2}\ldots$. Fix any $i\in\left\{  1,\ldots,\left|
\beta\right|  \right\}  $. Let $\eta$ be the finite path starting at $s\left(
\beta_{i}\right)  $ with range, $r\left(  \mu_{n_{1}}\right)  $, that is,
$\eta=\beta_{i}\beta_{i+1}\ldots\beta_{\left|  \beta\right|  }\mu_{n_{1}}$.
Since $x\cdot\mu_{n}\in C$ eventually, we may choose a finite path $\delta$
with $s\left(  \delta\right)  =s\left(  \mu_{n_{2}}\right)  $ and $v\in
r\left(  \delta\right)  $. So set $\theta=\eta\cdot\delta$, which belongs to
$\mathcal{G}^{\ast}$. Further, the equations, $s\left(  \theta\right)
=s\left(  \eta\right)  =s\left(  \beta_{i}\right)  $ and $v\in r\left(
\delta\right)  =r\left(  \theta\right)  $, hold. Hence $s\left(  \beta
_{i}\right)  \geq v$ for each $i\in\left\{  1,\ldots\left|  \beta\right|
\right\}  $, and since each $s\left(  \mu_{n_{j}}\right)  \geq v$, we may
conclude that $y\cdot\mu_{n}\in C$ eventually. So $y\cdot\mu\in F$ and
$\chi^{\prime}\in F$.

If $\left\vert \chi\right\vert <\infty$, then $r\left(  \chi\right)  =r\left(
\chi^{\prime}\right)  $, and a similar argument gives that $\chi^{\prime}\in
F$. We may use exactly the same argument to show that, if $\chi^{\prime}\in F$
then $\chi\in F$. Therefore $F$ is invariant.

We will contradict the assumption on the vertex $v$ by showing that if $\mu\in
F$ and $s\left(  \mu\right)  =v$, then $\mu=\gamma$. (For then, any sequence
in $F$ converging to $\gamma$ eventually will not have trivial isotropy.) We
can suppose that $\mu\in C$, since each $\mu^{\prime}\in F$ of finite length
and with $s\left(  \mu^{\prime}\right)  =v$ is the limit of sequence of such
an infinite path $\mu$.

Set $\mu=e_{1}^{\prime}e_{2}^{\prime}\ldots$ and for each $n$, let
$y=e_{1}^{\prime}\ldots e_{n}^{\prime}$. Since $\mu\in C$, $\ $we have
$s\left(  e_{n}^{\prime}\right)  \geq v$. So there is a finite path $\beta$
such that $s\left(  \beta\right)  =s\left(  e_{n}^{\prime}\right)  $ and $v\in
r\left(  \beta\right)  $. Note that $\alpha^{\prime}:=e_{1}^{\prime}\ldots
e_{n}^{\prime}\cdot\beta=y\cdot\beta$ is another loop based at $v$. Since
$\alpha$ is the only loop based at $v$, $\alpha^{\prime}$ is of the form
$\alpha\alpha\ldots\alpha$. It follows that every initial segment of $\mu$ is
an initial segment of $\gamma$. Therefore $\mu=\gamma$ and hence the
ultragraph $\mathcal{G}$ must satisfy condition (K).
\end{proof}

The following corollary is an immediate consequence of Theorem \ref{essp} and that fact that ultragraph groupoids are amenable.

\begin{theorem}\label{Ideals}
If $\mathcal{G}$ is an ultragraph satisfying condition (K), then the ideals in $C^*(\mathcal{G})$ are in bijective correspondence with the open invariant subsets of the unit space of $\mathfrak{G}_{\mathcal{G}}$.  In particular, $C^*(\mathcal{G})$ is simple if and only if $\mathfrak{G}_{\mathcal{G}}$ is minimal.
\end{theorem}

Recall that to say a groupoid is minimal is simply to say that the only invariant open subsets of its unit space are the empty set and the entire unit space.

\begin{proof}
Theorem \ref{essp} shows that $\mathfrak{G}_{\mathcal{G}}$ is essentially principal when $\mathcal{G}$ satisfies condition (K). On the other hand, Theorem \ref{amena} shows that all ultragraph groupoids are amenable and, therefore, that $C^*(\mathcal{G})\simeq C^*(\mathfrak{G}_{\mathcal{G}}) = C^*_{\rm{red}}(\mathfrak{G}_{\mathcal{G}})$.  Thus, the result follows from \cite[Proposition 2.4.6]{Rena}.
\end{proof}

We conclude with a groupoid approach to the sufficiency part of Theorem 3.11 in \cite{MTPS}, which gives necessary and sufficient conditions for an ultragraph $C^*$-algebra to be simple. First we
introduce the following definition.

\begin{definition}
Let $\mathcal{G}=\left(  G^{0},\mathcal{G}^{1},r,s\right)  $ be an ultragraph
and $v$ be a vertex. Let $A$ be a set in $\mathcal{G}^{0}$ and let $\alpha$ be
a finite path in $\mathcal{G}^{\ast}$. Then we write $v\longrightarrow
_{\alpha}A$, to mean that $s\left(  \alpha\right)  =v$ and $A\subseteq
r\left(  \alpha\right)  $. Roughly speaking the vertex $v$ reaches the set $A$
via one path $\alpha$. Compare with \cite[p.909]{MTPS}.
\end{definition}

\begin{theorem}
\label{simple} Let $\mathcal{G}=\left(  G^{0},\mathcal{G}^{1},r,s\right)  $ be
an ultragraph satisfying condition (K). Then the ultragraph $C^{\ast}$-algebra $C^{\ast}\left(
\mathcal{G}\right)  $ is simple if the following two conditions hold.

\begin{enumerate}

\item[$\mathit{(1)}$] $\mathcal{G}$ is cofinal (\cite{MTPS}) in the sense that
given a vertex $v$ and an infinite path $\gamma\in\mathfrak{p}^{\infty}$,
there exists an $n$ such that $v\geq s\left(  \gamma_{n}\right)  $; and

\item[$\mathit{(2)}$] if $A\in\mathcal{G}^{0}$ emits infinitely many edges in
$\mathcal{G}^{1}$, then for every $v\in G^{0}$ there exists a finite path
$\alpha\in\mathcal{G}^{\ast}$ such that $v\longrightarrow_{\alpha}A$.
\end{enumerate}
\end{theorem}

\begin{proof}
Suppose that $\mathcal{G}$ satisfies condition (K) and the two conditions \textit{(1)} and  \textit{(2)}. By the preceding comments, we just need to show that
$\mathfrak{G}_{\mathcal{G}}$ is minimal. Let $U\neq\emptyset$ be an open
invariant subset of $\mathfrak{G}_{\mathcal{G}}^{\left(  0\right)  }%
=Y_{\infty}\cup\mathfrak{p}^{\infty}$. Since $\mathfrak{p}^{\infty}$ is dense
in $Y_{\infty}\cup\mathfrak{p}^{\infty}$, the inequality, $U\cap\mathfrak
{p}^{\infty}\neq\emptyset$, holds. By considering a neighborhood in
$\mathfrak{G}_{\mathcal{G}}^{\left(  0\right)  }$ of some $\gamma\in
U\cap\mathfrak{p}^{\infty}$ we have $D_{\left(  \left(  \alpha,A\right)
,\left(  \alpha,A\right)  \right)  }\cap\mathfrak{G}_{\mathcal{G}}^{\left(
0\right)  }\subset U$ for some $\alpha\in\mathcal{G}^{\ast}$ and some
$A\in\mathcal{G}^{0}$ with $A\subseteq r\left(  \alpha\right)  $. (See Lemma
\ref{topoinf}.) Pick a vertex $v\in A$. Take any $\gamma=e_{1}e_{2}\ldots
\in\mathfrak{p}^{\infty}$. Then by \textit{(1),} $v\geq s\left(  e_{i}\right)
$ for some $i\in\mathbb{N}$. Then there is a finite path $\beta\in
\mathcal{G}^{\ast}$ such that $s\left(  \beta\right)  =v$ and $s\left(
e_{i}\right)  \in r\left(  \beta\right)  $. Notice that the triple, $\left(
e_{1}\ldots e_{i}e_{i+1}\ldots,\left|  e_{1}\ldots e_{i}\right|  -\left|
\alpha\beta e_{i}\right|  ,\alpha\beta e_{i}e_{i+1}\ldots\right)  $, belongs
to $\mathfrak{G}_{\mathcal{G}}$. Since $\alpha\beta e_{i}e_{i+1}\ldots\in
D_{\left(  \left(  \alpha,A\right)  ,\left(  \alpha,A\right)  \right)  }%
\cap\mathfrak{G}_{\mathcal{G}}^{\left(  0\right)  }\subset U$ and since $U$ is
invariant, we must have $\gamma\in U$. Next take any $y\in Y_{\infty}$. Then
the range of $y$, $r\left(  y\right)  $, is an infinite emitter. Then by
\textit{(2)} $v\longrightarrow_{\alpha^{\prime}}r\left(  y\right)  $ for some
$\alpha^{\prime}\in\mathcal{G}^{\ast}$. Thus $v=s\left(  \alpha^{\prime
}\right)  $ and $r\left(  y\right)  \subseteq r\left(  \alpha^{\prime}\right)
$. Observe that, the ultrapath, $(\alpha\alpha^{\prime},r\left(  y\right)  )$,
lies in $Y_{\infty}$. It follows, then, that the triple, $\left(
(\alpha\alpha^{\prime},r\left(  y\right)  ),\left|  \alpha\alpha^{\prime
}\right|  -\left|  y\right|  ,y\right)  $, belongs to $\mathfrak
{G}_{\mathcal{G}}$. (See Notation \ref{Notation}.) Since $(\alpha
\alpha^{\prime},r\left(  y\right)  )\in D_{\left(  \left(  \alpha,A\right)
,\left(  \alpha,A\right)  \right)  }\cap\mathfrak
{G}_{\mathcal{G}}^{\left(  0\right)  }\subset U$ and since $U$ is invariant,
we must have $y\in U$. Thus, the inclusion, $\mathfrak{G}_{\mathcal{G}%
}^{\left(  0\right)  }=Y_{\infty}\cup\mathfrak{p}^{\infty}\subseteq U$, holds,
and hence $\mathfrak{G}_{\mathcal{G}}^{\left(  0\right)  }=U$. So $\mathfrak
{G}_{\mathcal{G}}$ is minimal.
\end{proof}

\end{document}